\documentclass[12pt]{amsart}%
\usepackage{amsmath}
\usepackage{amsfonts}
\usepackage{amssymb}
\usepackage{graphicx}%

\providecommand{\U}[1]{\protect\rule{.1in}{.1in}}
\providecommand{\U}[1]{\protect \rule{.1in}{.1in}}
\providecommand{\U}[1]{\protect \rule{.1in}{.1in}}
\newtheorem{theorem}{Theorem}[section]
\newtheorem{lemma}{Lemma}[section]
\newtheorem{proposition}{Proposition}[section]
\newtheorem{corollary}{Corollary}[section]

\newtheorem{definition}{Definition}[section]

\numberwithin{equation}{section}
\newtheorem {conjecture}{Conjecture}

\theoremstyle{remark}
\newtheorem{remark}{Remark}[section]
\numberwithin{equation}{section}
\begin{document}
\title[On the CR analogue of Frankel conjecture ]{On the CR analogue of Frankel conjecture and a smooth representative of the
first Kohn-Rossi cohomology group }
\author{$^{^{\dag}}$Der-Chen Chang}
\address{$^{^{\dag}}$Department of Mathematics and Statistics, Georgetown University,
Washington D. C. 20057, USA\\
 \\
Graduate Institute of Business Administration, College of Management, Fu Jen
Catholic University, Taipei 242, Taiwan, R.O.C.}
\email{chang@georgetown.edu}
\author{$^{\ast}$Shu-Cheng Chang}
\address{$^{\ast}$Department of Mathematics and Taida Institute for Mathematical
Sciences (TIMS), National Taiwan University, Taipei 10617, Taiwan, R.O.C.}
\email{scchang@math.ntu.edu.tw }
\author{$^{\ast\ast}$Ting-Jung Kuo}
\address{$^{\ast\ast}$Department of Mathematics, National Taiwan Normal University,
Taipei 11677, Taiwan }
\email{tjkuo1215@ntnu.edu.tw}
\author{$^{^{\dag\dag}}$Chien Lin}
\address{$^{^{\dag\dag}}$Yau Mathematical Sciences Center, Tsinghua University, Haidian
District, Beijing 100084, China}
\email{chienlin@mail.tsinghua.edu.cn}
\thanks{$^{^{^{\dag}}}$Der-Chen Chang is partially supported by an NSF grant
DMS-1408839 and a McDevitt Endowment Fund at Georgetown University. }
\thanks{$^{\ast}$Shu-Cheng Chang and $^{\ast\ast}$Ting-Jung Kuo are partially
supported in part by the MOST of Taiwan.}

\begin{abstract}
In this note, we first give a criterion of pseudo-Einstein contact forms and
then affirm the CR analogue of Frankel conjecture in a closed, spherical,
strictly pseudoconvex CR manifold of nonnegative pseudohermitian curvature on
the space of smooth representatives of the first Kohn-Rossi cohomology group.
Moreover, we obtain the CR Frankel conjecture in a closed, spherical, strictly
pseudoconvex CR manifold with the vanishing first Kohn-Rossi cohomology group.
In particular, this conjecture holds in a spherical boundary of the Stein manifold.

\end{abstract}
\maketitle

\section{Introduction}

The well-known Riemann mapping theorem states that every simply connected
domain $\Omega$ properly contained in $\mathbb{C}$ is biholomorphically
equivalent to the open unit disc. In the paper of \cite{cj}, Chern and Ji
proved a generalization of the Riemann mapping theorem.

\begin{proposition}
\label{Pa} If $\Omega$ is a bounded, simply connected, strictly convex domain
in $\mathbb{C}^{n+1}$ and its connected smooth boundary $\partial\Omega$ has a
spherical CR structure, then it is biholomorphic to the unit ball and
$M=\partial\Omega$ is the standard CR $(2n+1)$-sphere.
\end{proposition}

It is also known from Burns and Shnider (\cite[Proposition 1.5.]{bs}) that if
$M$ is the compact \textbf{spherical }boundary of a Stein manifold, then
either $M$ is the standard CR sphere or $\pi_{1}(M)$ is infinite.

In Kaehler geometry, it was conjectured by Frankel (\cite{f}) that a closed
Kaehler manifold with positive bisectional curvature is biholomorphic to the
complex projective space. The Frankel conjecture was proved in later 1970s
independently by Mori (\cite{m}) and Siu-Yau (\cite{sy}). However Sasakian
geometry (that is, its pseudohermitian torsion tensor vanishes) is an odd
dimensional counterpart of Kaehler geometry, then it is natural to ask for
\textbf{CR analogue of Frankel conjecture} for Sasakian manifolds. In fact,
this is proved by He and Sun (\cite{hs}) :

\begin{proposition}
\label{Pb} The universal covering of any closed Sasakian $(2n+1)$-manifold of
positive pseudohermitian bisectional curvature must be CR equivalent to the
standard CR sphere $(\mathbf{S}^{2n+1},\widehat{J},\widehat{\theta})$.
\end{proposition}

Note that in view of Proposition \ref{Pb}, it involves the existence problem
of transversely Kaehler-Einstein metrics (pseudo-Einstein contact structures)
with positive pseudohermtian bisectional curvature and
\textbf{Sasakian-Eisntein} metrics in\ a closed Sasakian manifold.

From this inspiration, first by studying the existence theorem of
pseudo-Einstein contact structures in\ a closed, strictly pseudoconvex CR
$(2n+1)$-manifold of vanishing first Chern class for $n\geq2$ \ as in Theorem
\ref{T2} and Theorem \ref{T2a}, we are able to prove that such a manifold is
Sasakian when it is spherical with nonnegative pseudohermitian curvature on
the space of smooth representatives of the first Kohn-Rossi cohomology group.
Then we affirm the CR Frankel conjecture as in Theorem \ref{T13}, Corollary
\ref{C13}, Corollary \ref{C15}, and Theorem \ref{T14}. In particular, we
obtain the CR Frankel conjecture in a strictly pseudoconvex CR manifold which
is a spherical boundary of a Stein manifold.

More precisely, we first derive the key CR Bochner formulae as in Theorem
\ref{t31} which are involved the CR \ Paneitz operator. This is one of main
differences from Lee's key formula (\cite{l}) as in (\ref{2018AA}). By using
these formulae, we are able to obtain a pseudo-Eisntein contact form. Finally,
we prove that any closed, spherical, strictly pseudoconvex CR $(2n+1)$%
-manifold $(M,J,\theta)$ of pseudo-Eisntein contact form $\theta$ with the
positive constant Tanaka-Webster scalar curvature $R$ must be \textbf{Sasakian
space form} and manifolds always admit Riemannian metrics with positive Ricci
curvature (\cite{cc}), so they must have finite fundamental group and the
manifolds is a finite quotent of a standard CR sphere (\cite{t}). Therefore
the universal covering of $M$ is globally CR equivalent to a standard CR sphere.

A strictly pseudoconvex CR $(2n+1)$-manifold is called
\textit{pseudo-Einstein} if its pseudohermitian Ricci curvature tensor is
function-proportional to its Levi metric%
\begin{equation}
R_{\alpha\overline{\beta}}=\frac{R}{n}h_{\alpha\overline{\beta}} \label{2019}%
\end{equation}
for $n\geq2$. It is equivalent to saying the following quantity is vanishing
(\cite{l}, \cite{h}, \cite{ckl})
\begin{equation}
W_{\alpha}\doteqdot\left(  R,_{\alpha}-inA_{\alpha\beta},^{\beta}\right)
=0\text{.} \label{2019x}%
\end{equation}
Then the pseudo-Einstein condition (\ref{2019}) can be replaced by
(\ref{2019x}) for any $n\geq1.$ This is the main different point view from the
previous work by J. Lee (\cite{l}). Here we come out with several key
Bochner-type formulae \ as in Theorem \ref{t31}.

From this, we define (\cite{h}, \cite{fh}, \cite{ccc}) the CR analogue of
$Q$-curvature by%

\begin{equation}
Q:=-\operatorname{Re}[\left(  R,_{\alpha}-inA_{\alpha\beta,\overline{\beta}%
}\right)  _{\overline{\alpha}}]=-\frac{1}{2}(W_{\alpha\overline{\alpha}%
}+W_{\overline{\alpha}\alpha}). \label{c}%
\end{equation}
Lee (\cite{l}) showed an obstruction to the existence of a pseudo-Einstein
contact form $\theta$ which is the vanishing of first Chern class
$c_{1}(T_{1,0}M)$ for a closed, strictly pseudoconvex $(2n+1)$-manifold
$\left(  M,J,\theta\right)  $ with $n\geq2$. Thereafter, Lee conjectured that

\begin{conjecture}
\label{conj1} Any closed, strictly pseudoconvex CR $(2n+1)$-manifold of the
vanishing first Chern class $c_{1}(T_{1,0}M)$ admits a global pseudo-Einstein
structure for $n\geq2$.
\end{conjecture}

To set up the method, we recall J. J. Kohn's Hodge theory for the
$\overline{\partial}_{b}$ complex (\cite{k}). Let $(M,J,\theta)$ be a closed,
strictly pseudoconvex CR $(2n+1)$-manifold and $\eta\in\Omega^{0,1}\left(
M\right)  $ a smooth $\left(  0,1\right)  $-form on $M$ with
\[
\overline{\partial}_{b}\eta=0.
\]
Then there exists a smooth complex-valued function $\varphi=u+iv\in C_{%
\mathbb{C}
}^{\infty}\left(  M\right)  $ and a smooth $\left(  0,1\right)  $-form
$\gamma\in\Omega^{0,1}\left(  M\right)  $ for $\gamma=\gamma_{\overline
{\alpha}}\theta^{\overline{\alpha}}$ such that
\begin{equation}
\left(  \eta-\overline{\partial}_{b}\varphi\right)  =\gamma\in\ker\left(
\square_{b}\right)  , \label{2019C}%
\end{equation}
where $\square_{b}=2\left(  \overline{\partial}_{b}\overline{\partial}%
_{b}^{\ast}+\overline{\partial}_{b}^{\ast}\overline{\partial}_{b}\right)  $ is
the Kohn-Rossi Laplacian.

Let the first Chern class $c_{1}(T^{1,0}M)$ of $T^{1,0}M$ be represented by
$\Theta$ with
\[
c_{1}(T^{1,0}M)=\frac{i}{2\pi}[d\omega_{\alpha}{}^{\alpha}]=\frac{i}{2\pi
}[\Theta]
\]
and
\[
\Theta=R_{\alpha\overline{\beta}}\theta^{\alpha}\wedge\theta^{\overline{\beta
}}+A_{\mu\alpha,\overline{\alpha}}{}\theta^{\mu}\wedge\theta-A_{\overline{\mu
}\overline{\alpha},\alpha}\theta^{\overline{\mu}}\wedge\theta,
\]
which is the purely imaginary two-form. In this paper, \ we assume
$c_{1}(T_{1,0}M)=0.$ Then there is a pure imaginary $1$-form
\[
\sigma=\sigma_{\overline{\alpha}}\theta^{\overline{\alpha}}-\sigma_{\alpha
}\theta^{\alpha}+i\sigma_{0}\theta
\]
with
\begin{equation}
d\omega_{\alpha}^{\alpha}=d\sigma=\Theta\label{2019aaa}%
\end{equation}
for the pure imaginary Webster connection form $\omega_{\alpha}^{\alpha}.$ As
in Lemma \ref{l32}, we choose the $\left(  0,1\right)  $-form $\eta\in
\Omega^{0,1}\left(  M\right)  $
\[
\eta=\sigma_{\overline{\alpha}}\theta^{\overline{\alpha}}.
\]
Then $\sigma_{\overline{\alpha}}\theta^{\overline{\alpha}}$ is $\overline
{\partial}_{b}$-closed and the Kohn-Rossi solution is
\begin{equation}
\varphi_{\overline{\alpha}}=\sigma_{\overline{\alpha}}-\gamma_{\overline
{\alpha}}. \label{20aa}%
\end{equation}
By combining the CR Bochner-type estimates as in Theorem \ref{t31}, we are
able to prove the existence theorem of pseudo-Einstein contact structures
$\widetilde{\theta}=e^{\frac{2u}{n+2}}\theta$ in\ a closed, strictly
pseudoconvex CR $(2n+1)$-manifold of vanishing first Chern class as in Theorem
\ref{T2} and Theorem \ref{T2a} for $n\geq2$. \ However, it follows from
(\ref{4aa}), (\ref{01}), (\ref{2020b}), and (\ref{2019b}) that $\theta$ is
also a pseudo-Einstein contact structure only if its CR $Q$-curvature is CR-pluriharmonic.

Therefore by inspirations from Theorem \ref{T51}, Lee Conjecture \ref{conj1},
and results as in \cite{cj}, \cite{bs} and \cite{hs}, we make the following CR
analogue Frankel conjecture :

\begin{conjecture}
\label{conj2} ( \textbf{CR Frankel Conjecture}) \ Let $(M,J,\theta)$ be a
closed, spherical, strictly pseudoconvex CR $(2n+1)$-manifold of the vanishing
first Chern class $c_{1}(T_{1,0}M),n\geq2.$ Then the universal covering of $M$
is CR equivalent to the standard CR sphere $(\mathbf{S}^{2n+1},\widehat
{J},\widehat{\theta})$ if $\theta$ has the positive constant Tanaka-Webster
scalar curvature $R$ and its CR $Q$-curvature is CR-pluriharmonic.
\end{conjecture}

Now we are ready to apply results as in section $4$ and section $5$ to affirm
the CR analogue of Frankel conjecture via the nonnegativity of pseudohermitian
curvature as in (\ref{2019B}) and smooth representative of the first
Kohn-Rossi cohomology group. In fact, as a consequence of Theorem \ref{T2} and
Theorem \ref{T51}, we have

\begin{theorem}
\label{T13} Let $(M,J,\theta)$ be a closed, spherical, strictly pseudoconvex
CR $(2n+1)$-manifold of $c_{1}(T^{1,0}M)=0,$ $n\geq2$. Suppose that
\begin{equation}
(Ric-\frac{1}{2}Tor)\left(  \rho,\rho\right)  \geq0 \label{2019B}%
\end{equation}
on the space of smooth representatives $(0,1)$-form $\rho=\rho_{\overline
{\alpha}}\theta^{\overline{\alpha}}\in\Omega^{0,1}\left(  M\right)  $ of the
first Kohn-Rossi cohomology group $H_{\overline{\partial}_{b}}^{0,1}(M)$ (i.e.
$\rho\in\ker\left(  \square_{b}\right)  )$. Then the universal covering of $M$
is CR equivalent to the standard CR sphere $(\mathbf{S}^{2n+1},\widehat
{J},\widehat{\theta})$ if $\theta$ has the positive constant Tanaka-Webster
scalar curvature $R$ and the CR-pluriharmonic\ $Q$-curvature. Here
$Ric(\rho,\rho)=R_{\alpha\overline{\beta}}\rho_{\overline{\alpha}}\rho_{\beta
}$ and $Tor\left(  \rho,\rho\right)  :=i(A_{\overline{\alpha}\overline{\beta}%
}\rho_{\alpha}\rho_{\beta}-A_{\alpha\beta}\rho_{\overline{\alpha}}%
\rho_{\overline{\beta}}).$
\end{theorem}

We observe that the pseudohermitian curvature quantity (\ref{2019B}) appears
in the CR Bochner formula (\ref{2020ABC}) as in the paper \cite{cc}.

In particular, as\ a consequence of Lemma \ref{l35}, Proposition \ref{P1}, and
Theorem \ref{T13}, we have

\begin{corollary}
\label{C13} Let $(M,J,\theta)$ be a closed, spherical, strictly pseudoconvex
CR $(2n+1)$-manifold of $c_{1}(T^{1,0}M)=0,$ $n\geq2$ with either

(i) vanishing of the first Kohn-Rossi cohomology group $H_{\overline{\partial
}_{b}}^{0,1}(M)$, or

(ii)
\[
Ric\left(  \rho,\rho\right)  \geq0
\]
on the space of smooth representatives $\rho$ of the first Kohn-Rossi
cohomology group $H_{\overline{\partial}_{b}}^{0,1}(M).$

Then the universal covering of $M$ is CR equivalent to the standard CR sphere
$(\mathbf{S}^{2n+1},\widehat{J},\widehat{\theta})$ if $\theta$ has the
positive constant Tanaka-Webster scalar curvature $R$ and the
CR-pluriharmonic\ $Q$-curvature.
\end{corollary}

Let $(M,J,\theta)$ be a closed, strictly pseudoconvex CR $(2n+1)$-manifold in
the boundary of a bounded strongly pseudoconvex domain $D$ in $\mathbb{C}%
^{n+2},n\geq2$. In the paper of \cite[Theorem C.]{y}, Yau proved that $M$ is a
boundary of the complex sub-manifold $V$ of $D\setminus M$ if and only if
Kohn-Rossi cohomology groups $H_{\overline{\partial}_{b}}^{p,q}(M)$ are zero
for $1\leq q\leq n-1$. Then as a conseqence of Corollary \ref{C13}, we have

\begin{corollary}
\label{C15} Let $(M,J,\theta)$ be a closed, spherical, strictly pseudoconvex
CR $(2n+1)$-manifold of $c_{1}(T^{1,0}M)=0,$ $n\geq2$ in the boundary of a
bounded strongly pseudoconvex domain $D$ in $\mathbb{C}^{n+2}$. Assume that
$M$ is a boundary of the complex sub-manifold $V$ of $D\setminus M$. Then the
universal covering of $M$ is CR equivalent to the standard CR sphere
$(\mathbf{S}^{2n+1},\widehat{J},\widehat{\theta})$ if $\theta$ has the
positive constant Tanaka-Webster scalar curvature $R$ and the
CR-pluriharmonic\ $Q$-curvature.
\end{corollary}

Furthermore, as a consequence of Theorem \ref{T2a} and Theorem \ref{T51}, we have

\begin{theorem}
\label{T14} Let $(M,J,\theta)$ be a closed, spherical, strictly pseudoconvex
CR $(2n+1)$-manifold of $c_{1}(T^{1,0}M)=0,$ $n\geq2$ with $d\omega_{\alpha
}^{\alpha}=d\sigma,\sigma=\sigma_{\overline{\alpha}}\theta^{\overline{\alpha}%
}-\sigma_{\alpha}\theta^{\alpha}+i\sigma_{0}\theta.$ Assume that $\eta
=\sigma_{\overline{\alpha}}\theta^{\overline{\alpha}}$ satisfies

(i)
\[
\eta\in\ker\left(  \square_{b}\right)  ,
\]

(ii)%
\[
Tor^{\prime}\left(  \eta,\eta\right)  =0.
\]
Then the universal covering of $M$ is CR equivalent to the standard CR sphere
$(\mathbf{S}^{2n+1},\widehat{J},\widehat{\theta})$ if $\theta$ has the
positive constant Tanaka-Webster scalar curvature $R$.
\end{theorem}

In particular, we have

\begin{corollary}
\label{C14} Let $(M,J,\theta)$ be a closed, spherical, strictly pseudoconvex
CR $(2n+1)$-manifold of $c_{1}(T^{1,0}M)=0,$ $n\geq2$ with $d\omega_{\alpha
}^{\alpha}=d\sigma,\sigma=\sigma_{\overline{\alpha}}\theta^{\overline{\alpha}%
}-\sigma_{\alpha}\theta^{\alpha}+i\sigma_{0}\theta.$ Assume that
\[
d\sigma=id(f\theta),
\]
for some smooth, real-valued function $f$. \ Then the universal covering of
$M$ is CR equivalent to the standard CR sphere $(\mathbf{S}^{2n+1},\widehat
{J},\widehat{\theta})$ if $\theta$ has the positive constant Tanaka-Webster
scalar curvature $R$.
\end{corollary}

\begin{remark}
(i) By the contracted Bianchi identity, (\ref{c}), and (\ref{2019b}), then the
CR-pluriharmonic\ $Q$-curvature is equivalent to
\[
A_{\alpha\beta,\overline{\alpha}\overline{\beta}}=0
\]
as in Theorem \ref{T13}, Corollary \ref{C13}, and Corollary \ref{C15}.

(i) For $n=1$, we refer to the authors' previous work where one needs the
positivity condition of the CR Paneitz operator in a closed spherical strictly
pseudoconvex CR $3$-manifold as in \cite{ckl}.
\end{remark}

We briefly describe the methods used in our proofs. In section $2$, we
introduce some basic materials in a pseudohermitian $(2n+1)$-manifold. In
section $3$, we will derive some crucial results such as the CR Bochner-type
formula. In section $4$, we give the existence theorems of pseudo-Einstein
contact structures. In the final section, we then affirm the CR Frankel
conjecture in a closed, spherical, strictly pseudoconvex CR $(2n+1)$-manifold.

\noindent\textbf{Acknowledgements} Part of the project was done during
visiting to Yau Mathematical Sciences Center, Tsinghua University. The last
three named authors would like to express their thanks for those warm
hospitality there. We also thank Professor Yuya Takeuchi for very useful comments.

\section{Preliminaries}

In this section, we recall some ingredients needed to prove main results in
this paper. We first introduce some basic materials in a pseudohermitian
$(2n+1)$-manifold (see \cite{l}). Let $(M,\xi)$ be a $(2n+1)$-dimensional,
orientable, contact manifold with contact structure $\xi$. A CR structure
compatible with $\xi$ is an endomorphism $J:\xi\rightarrow\xi$ such that
$J^{2}=-1$. We also assume that $J$ satisfies the following integrability
condition: If $X$ and $Y$ are in $\xi$, then so are $[JX,Y]+[X,JY]$ and
$J([JX,Y]+[X,JY])=[JX,JY]-[X,Y]$.

Let $\left\{  T,Z_{\alpha},Z_{\bar{\alpha}}\right\}  $ be a frame of
$TM\otimes\mathbb{C}$, where $Z_{\alpha}$ is any local frame of $T_{1,0}%
,\ Z_{\bar{\alpha}}=\overline{Z_{\alpha}}\in T_{0,1},$ and $T$ is the
characteristic vector field. Then $\left\{  \theta,\theta^{\alpha}%
,\theta^{\bar{\alpha}}\right\}  $, which is the coframe dual to $\left\{
T,Z_{\alpha},Z_{\bar{\alpha}}\right\}  $, satisfies
\begin{equation}
d\theta=ih_{\alpha\overline{\beta}}\theta^{\alpha}\wedge\theta^{\overline
{\beta}} \label{72}%
\end{equation}
for some positive definite hermitian matrix of functions $(h_{\alpha\bar
{\beta}})$. We also call such $M$ a strictly pseudoconvex CR $(2n+1)$%
-manifold. The Levi form $\left\langle \ ,\ \right\rangle _{L_{\theta}}$ is
the Hermitian form on $T_{1,0}$ defined by%
\[
\left\langle Z,W\right\rangle _{L_{\theta}}=-i\left\langle d\theta
,Z\wedge\overline{W}\right\rangle .
\]
We can extend $\left\langle \ ,\ \right\rangle _{L_{\theta}}$ to $T_{0,1}$ by
defining $\left\langle \overline{Z},\overline{W}\right\rangle _{L_{\theta}%
}=\overline{\left\langle Z,W\right\rangle }_{L_{\theta}}$ for all $Z,W\in
T_{1,0}$. The Levi form naturally induces a Hermitian form on the dual bundle
of $T_{1,0}$, denoted by $\left\langle \ ,\ \right\rangle _{L_{\theta}^{\ast}%
}$, and hence on all the induced tensor bundles. Integrating the Hermitian
form (when acting on sections) over $M$ with respect to the volume form
$d\mu=\theta\wedge(d\theta)^{n}$, we get an inner product on the space of
sections of each tensor bundle.

The pseudohermitian connection of $(J,\theta)$ is the connection $\nabla$ on
$TM\otimes\mathbb{C}$ (and extended to tensors) given in terms of a local
frame $Z_{\alpha}\in T_{1,0}$ by%

\[
\nabla Z_{\alpha}=\omega_{\alpha}{}^{\beta}\otimes Z_{\beta},\quad\nabla
Z_{\bar{\alpha}}=\omega_{\bar{\alpha}}{}^{\bar{\beta}}\otimes Z_{\bar{\beta}%
},\quad\nabla T=0,
\]
where $\omega_{\alpha}{}^{\beta}$ are the $1$-forms uniquely determined by the
following equations:%

\[%
\begin{split}
d\theta^{\beta}  &  =\theta^{\alpha}\wedge\omega_{\alpha}{}^{\beta}%
+\theta\wedge\tau^{\beta},\\
0  &  =\tau_{\alpha}\wedge\theta^{\alpha},\\
0  &  =\omega_{\alpha}{}^{\beta}+\omega_{\bar{\beta}}{}^{\bar{\alpha}}.
\end{split}
\]
We can write (by the Cartan lemma) $\tau_{\alpha}=A_{\alpha\gamma}%
\theta^{\gamma}$ with $A_{\alpha\gamma}=A_{\gamma\alpha}$. The curvature of
Tanaka-Webster connection, expressed in terms of the coframe $\{\theta
=\theta^{0},\theta^{\alpha},\theta^{\bar{\alpha}}\}$, is
\[%
\begin{split}
\Pi_{\beta}{}^{\alpha}  &  =\overline{\Pi_{\bar{\beta}}{}^{\bar{\alpha}}%
}=d\omega_{\beta}{}^{\alpha}-\omega_{\beta}{}^{\gamma}\wedge\omega_{\gamma}%
{}^{\alpha},\\
\Pi_{0}{}^{\alpha}  &  =\Pi_{\alpha}{}^{0}=\Pi_{0}{}^{\bar{\beta}}=\Pi
_{\bar{\beta}}{}^{0}=\Pi_{0}{}^{0}=0.
\end{split}
\]
Webster showed that $\Pi_{\beta}{}^{\alpha}$ can be written
\[
\Pi_{\beta}{}^{\alpha}=R_{\beta}{}^{\alpha}{}_{\rho\bar{\sigma}}\theta^{\rho
}\wedge\theta^{\bar{\sigma}}+W_{\beta}{}^{\alpha}{}_{\rho}\theta^{\rho}%
\wedge\theta-W^{\alpha}{}_{\beta\bar{\rho}}\theta^{\bar{\rho}}\wedge
\theta+i\theta_{\beta}\wedge\tau^{\alpha}-i\tau_{\beta}\wedge\theta^{\alpha}%
\]
where the coefficients satisfy
\[
R_{\beta\bar{\alpha}\rho\bar{\sigma}}=\overline{R_{\alpha\bar{\beta}\sigma
\bar{\rho}}}=R_{\bar{\alpha}\beta\bar{\sigma}\rho}=R_{\rho\bar{\alpha}%
\beta\bar{\sigma}},\ \ \ W_{\beta\bar{\alpha}\gamma}=W_{\gamma\bar{\alpha
}\beta}.
\]
Here $R_{\gamma}{}^{\delta}{}_{\alpha\bar{\beta}}$ is the pseudohermitian
curvature tensor, $R_{\alpha\bar{\beta}}=R_{\gamma}{}^{\gamma}{}_{\alpha
\bar{\beta}}$ is the pseudohermitian Ricci curvature tensor and $A_{\alpha
\beta}$\ is the pseudohermitian torsion tensor. Furthermore, we denote \
\[
Tor(X,Y):=h^{\alpha\bar{\beta}}T_{\alpha\overline{\beta}}(X,Y)=i(A_{\overline
{\alpha}\bar{\rho}}X^{\overline{\rho}}Y^{\overline{\alpha}}-A_{\alpha\rho
}X^{\rho}Y^{\alpha})
\]
for any $X=X^{\alpha}Z_{\alpha},\ Y=Y^{\alpha}Z_{\alpha}$ in $T_{1,0}.$ We
will denote components of covariant derivatives with indices preceded by
comma; thus write $A_{\alpha\beta,\gamma}$. The indices $\{0,\alpha
,\bar{\alpha}\}$ indicate derivatives with respect to $\{T,Z_{\alpha}%
,Z_{\bar{\alpha}}\}$. For derivatives of a scalar function, we will often omit
the comma, for instance, $u_{\alpha}=Z_{\alpha}u,\ u_{\alpha\bar{\beta}%
}=Z_{\bar{\beta}}Z_{\alpha}u-\omega_{\alpha}{}^{\gamma}(Z_{\bar{\beta}%
})Z_{\gamma}u.$ For a smooth real-valued function $u$, the subgradient
$\nabla_{b}$ is defined by $\nabla_{b}u\in\xi$ and $\left\langle Z,\nabla
_{b}u\right\rangle _{L_{\theta}}=du(Z)$ for all vector fields $Z$ tangent to
the contact plane. Locally, we denote $\nabla_{b}u=\sum_{\alpha}u_{\bar
{\alpha}}Z_{\alpha}+u_{\alpha}Z_{\bar{\alpha}}$. We also denote $u_{0}=Tu$. We
can use the connection to define the subhessian as the complex linear map
$(\nabla^{H})^{2}u:T_{1,0}\oplus T_{0,1}\rightarrow T_{1,0}\oplus T_{0,1}$ by
\[
(\nabla^{H})^{2}u(Z)=\nabla_{Z}\nabla_{b}u.
\]
In particular,%

\[%
\begin{array}
[c]{c}%
|\nabla_{b}u|^{2}=2\sum_{\alpha}u_{\alpha}u_{\overline{\alpha}},\quad
|\nabla_{b}^{2}u|^{2}=2\sum_{\alpha,\beta}(u_{\alpha\beta}u_{\overline{\alpha
}\overline{\beta}}+u_{\alpha\overline{\beta}}u_{\overline{\alpha}\beta}).
\end{array}
\]
Also
\[%
\begin{array}
[c]{c}%
\Delta_{b}u=Tr\left(  (\nabla^{H})^{2}u\right)  =\sum_{\alpha}(u_{\alpha
\bar{\alpha}}+u_{\bar{\alpha}\alpha}).
\end{array}
\]

\begin{definition}
\label{d41} (\cite{l}, \cite{cj}) Let $(M,\theta)$ be a closed strictly
pseudoconvex CR $(2n+1)$-manifold with $n\geq2.$

(i) We define the first Chern class $c_{1}(T_{1,0}M)\in H^{2}(M,\mathbf{R})$
for the holomorphic tangent bundle $T^{1,0}M$ by
\begin{align}
c_{1}(T^{1,0}M)  &  =\frac{i}{2\pi}[d\omega_{\alpha}{}^{\alpha}]
\label{202020}\\
&  =\frac{i}{2\pi}[R_{\alpha\overline{\beta}}\theta^{\alpha}\wedge
\theta^{\overline{\beta}}+A_{\alpha\mu,\overline{\alpha}}{}\theta^{\mu}%
\wedge\theta-A_{\overline{\alpha}\overline{\mu},\alpha}\theta^{\overline{\mu}%
}\wedge\theta].\nonumber
\end{align}

(ii) We call a CR structure $J$ spherical if the Chern curvature tensor
\begin{equation}%
\begin{array}
[c]{ccl}%
C_{\beta\overline{\alpha}\lambda\overline{\sigma}} & = & R_{\beta
\overline{\alpha}\lambda\overline{\sigma}}-\frac{1}{n+2}[R_{\beta
\overline{\alpha}}h_{\lambda\overline{\sigma}}+R_{\lambda\overline{\alpha}%
}h_{\beta\overline{\sigma}}+\delta_{\beta}^{\alpha}R_{\lambda\overline{\sigma
}}+\delta_{\lambda}^{\alpha}R_{\beta\overline{\sigma}}]\\
&  & +\frac{R}{(n+1)(n+2)}[\delta_{\beta}^{\alpha}h_{\lambda\overline{\sigma}%
}+\delta_{\lambda}^{\alpha}h_{\beta\overline{\sigma}}]
\end{array}
\label{d12}%
\end{equation}
vanishes identically.
\end{definition}

\begin{remark}
1. Note that $C_{\alpha\overline{\alpha}\lambda\overline{\sigma}}=0.$ Hence
$C_{\beta\overline{\alpha}\lambda\overline{\sigma}}$ is always vanishing for
$n=1.$

2. We observe that the spherical structure is CR invariant and a closed
spherical CR $(2n+1)$-manifold $(M,J)$ is locally CR equivalent to
$(\mathbf{S}^{2n+1},\widehat{J}).$

3. (\cite{kt}) In general, a\ spherical CR structure on a $(2n+1)$-manifold is
a system of coordinate charts into $S^{2n+1}$\ such that the overlap functions
are restrictions of elements of\textrm{\ }$PU(n+1,1)$. Here $PU(n+1,1)$ is the
group of complex projective automorphisms of the unit ball in $\mathbf{C}%
^{n+1}$ and the holomorphic isometry group of the complex hyperbolic space
$\mathbf{CH}^{n}$.
\end{remark}

\begin{definition}
(i) Let $(M,\xi,\theta)$ be a closed pseudohermitian $(2n+1)$-manifold.
Define
\[%
\begin{array}
[c]{c}%
P\varphi=\sum_{\alpha=1}^{n}(\varphi_{\overline{\alpha}}{}^{\overline{\alpha}%
}{}_{\beta}+inA_{\beta\alpha}\varphi^{\alpha})\theta^{\beta}=(P_{\beta}%
\varphi)\theta^{\beta},\text{ }\beta=1,2,\cdot\cdot\cdot,n
\end{array}
\]
which is an operator that characterizes CR-pluriharmonic functions (\cite{l}
for $n=1$ and \cite{gl} for $n\geq2$). Here $P_{\beta}\varphi=\sum_{\alpha
=1}^{n}(\varphi_{\overline{\alpha}}{}^{\overline{\alpha}}{}_{\beta}%
+inA_{\beta\alpha}\varphi^{\alpha})$ and $\overline{P}\varphi=(\overline
{P}_{\beta}\varphi)\theta^{\overline{\beta}}$, the conjugate of $P$. Moreover,
we define
\begin{equation}
P_{0}\varphi=\delta_{b}(P\varphi)+\overline{\delta}_{b}(\overline{P}\varphi)
\label{ABC3}%
\end{equation}
which is the so-called CR Paneitz operator $P_{0}.$ Here $\delta_{b}$ is the
divergence operator that takes $(1,0)$-forms to functions by $\delta
_{b}(\sigma_{\alpha}\theta^{\alpha})=\sigma_{\alpha},^{\alpha}.$ Hence $P_{0}$
is a real and symmetric operator and
\[%
\begin{array}
[c]{c}%
\int_{M}\langle P\varphi+\overline{P}\varphi,d_{b}\varphi\rangle_{L_{\theta
}^{\ast}}d\mu=-\int_{M}\left(  P_{0}\varphi\right)  \varphi d\mu.
\end{array}
\]

(ii) We call the Paneitz operator $P_{0}$ with respect to $(J,\theta)$
essentially positive if there exists a constant $\Lambda$ $>$ $0$ such that
\begin{equation}
\int_{M}P_{0}\varphi\cdot\varphi d\mu\geq\Lambda\int_{M}\varphi^{2}d\mu.
\label{41}%
\end{equation}
for all real smooth functions $\varphi$ $\in(\ker P_{0})^{\perp}$ (i.e.
perpendicular to the kernel of $P_{0}$ in the $L^{2}$ norm with respect to the
volume form $d\mu$ $=$ $\theta\wedge d\theta).$ We say that $P_{0}$ is
nonnegative if
\[
\int_{M}P_{0}\varphi\cdot\varphi d\mu\geq0
\]
for all real smooth functions $\varphi$.
\end{definition}

\begin{remark}
\label{r1} 1. The space of kernel of the CR Paneitz operator $P_{0}$ is
infinite dimensional, containing all $CR$ -pluriharmonic functions. However,
for a closed pseudohermitian $(2n+1)$-manifold $(M,\xi,\theta)$ with $n\geq2$,
it was shown (\cite{gl}) that
\begin{equation}
\ker P_{\beta}=\ker P_{0}. \label{4aa}%
\end{equation}
2. (\cite{gl}, \cite{cc}) The CR Paneitz $P_{0}$ is always nonnegative for a
closed pseudohermitian $(2n+1)$-manifold $(M,\xi,\theta)$ with $n\geq2$.

3. (\cite{l}) A real-valued smooth function $u$ is said to be CR-pluriharmonic
if, for any point $x\in M$, there is a real-valued smooth function $v$ such
that
\begin{equation}
\overline{\partial}_{b}(u+iv)=0. \label{20a}%
\end{equation}

\end{remark}

\section{The Bochner-Type Formulae}

In this section, we first derive some essential lemmas. Recall that the
transformation law of the connection under a change of pseudohermitian
structure was computed in \cite[Sec. 5]{l2}. Let $\hat{\theta}=e^{2f}\theta$
be another pseudohermitian structure. Then we can define an admissible coframe
by $\hat{\theta}^{\alpha}=e^{f}(\theta^{\alpha}+2if^{\alpha}\theta)$. With
respect to this local coframe, the connection $1$-form and the pseudohermitian
torsion are given by%
\begin{equation}%
\begin{split}
\widehat{{\omega}}{_{\beta}}^{\alpha}  &  ={\omega_{\beta}}^{\alpha
}+2(f_{\beta}\theta^{\alpha}-f^{\alpha}\theta_{\beta})+\delta_{\beta}^{\alpha
}(f_{\gamma}\theta^{\gamma}-f^{\gamma}\theta_{\gamma})\\
&  \phantom{=}+i(f^{\alpha}{}_{\beta}+f_{\beta}{}^{\alpha}+4\delta_{\beta
}^{\alpha}f_{\gamma}f^{\gamma})\theta,
\end{split}
\label{30}%
\end{equation}
and
\begin{equation}
\widehat{{A}}{_{\alpha\beta}=}e^{-2f}({A_{\alpha\beta}+2i}f_{\alpha\beta
}-4if_{\alpha}f_{\beta}), \label{31}%
\end{equation}
respectively. Thus the Webster curvature transforms as
\begin{equation}
\widehat{R}=e^{-2f}(R-2(n+1)\Delta_{b}f-4n(n+1)f_{\gamma}f^{\gamma}).
\label{32}%
\end{equation}
Here covariant derivatives on the right side are taken with respect to the
pseudohermitian structure $\theta$ and an admissible coframe $\theta^{\alpha}%
$. Note also that the dual frame of $\{\hat{\theta},\hat{\theta}^{\alpha}%
,\hat{\theta}^{\overline{\alpha}}\}$ is given by $\{\widehat{T},\widehat
{Z}_{\alpha},\widehat{Z}_{\overline{\alpha}}\}$, where%

\[
\widehat{T}=e^{-2f}(T+2if^{\overline{\gamma}}Z_{\overline{\gamma}}%
-2if^{\gamma}Z_{\gamma}),\text{ \ }\widehat{Z}_{\alpha}=e^{-f}Z_{\alpha}.
\]

Now we derive the following transformation property for the CR-pluriharmonic
operator and CR Paneitz operator.

\begin{lemma}
\label{l21} Let $\theta$ and $\widehat{\theta}$ be contact forms in a
$(2n+1)$-dimensional pseudohermitian manifold $(M,\xi)$. If $\ \widehat
{\theta}=e^{2f}\theta,$ then we have%
\begin{equation}%
\begin{array}
[c]{l}%
\widehat{R}_{\alpha}-in\widehat{A}_{\alpha\beta},^{\beta}=e^{-3f}[R_{\alpha
}-inA_{\alpha\beta},^{\beta}-2(n+2)P_{\alpha}f]\\
\ \ \ \ \ \ \ \ \ \ \ \ \ \ \ \ \ \ \ \ \ \ \ \ \ \ \ +2ne^{-2f}(\widehat
{R}_{\alpha\overline{\beta}}-\frac{\widehat{R}}{n}\widehat{h}_{\alpha
\overline{\beta}})f^{\overline{\beta}}.
\end{array}
\label{4}%
\end{equation}

\end{lemma}

\begin{proof}
By the contracted Bianchi identity, we have
\[%
\begin{array}
[c]{c}%
\frac{n-1}{n}(R_{\alpha}-inA_{\alpha\beta},^{\beta})=(R_{\alpha\overline
{\beta}}-\frac{R}{n}h_{\alpha\overline{\beta}}),^{\overline{\beta}}.
\end{array}
\]
Also, by \cite[P 172]{l2}
\begin{equation}%
\begin{array}
[c]{c}%
(R_{\alpha\overline{\beta}}-\frac{R}{n}h_{\alpha\overline{\beta}%
})-2(n+2)(f_{\alpha\overline{\beta}}-\frac{1}{n}f_{\gamma}{}^{\gamma}%
h_{\alpha\overline{\beta}})=\widehat{R}_{\alpha\overline{\beta}}%
-\frac{\widehat{R}}{n}\widehat{h}_{\alpha\overline{\beta}}.
\end{array}
\label{5}%
\end{equation}
Following the same computation as the proof of Lemma 5.4 in \cite{h}, by using
\eqref{30}, \eqref{31}, and \eqref{32}, we compute%
\begin{align*}
\widehat{R}_{\alpha}  &  =\widehat{Z}_{\alpha}\widehat{R}=e^{-f}Z_{\alpha
}e^{-2f}(R-2(n+1)\Delta_{b}f-2n(n+1)|\nabla_{b}f|^{2})\\
&  =e^{-3f}[R_{\alpha}-2Wf_{\alpha}+4(n+1)(\Delta_{b}f+n|\nabla_{b}%
f|^{2})f_{\alpha}\\
&  \text{ \ \ \ \ }-2(n+1)(f_{\gamma}{}^{\gamma}{}_{\alpha}+f_{\overline
{\gamma}}{}^{\overline{\gamma}}{}_{\alpha})-4n(n+1)(f_{\gamma\alpha}f^{\gamma
}+f_{\gamma}f^{\gamma}{}_{\alpha})],\\
i\widehat{A}_{\alpha\beta},_{\overline{\gamma}}  &  =i(\widehat{Z}%
_{\overline{\gamma}}\widehat{A}_{\alpha\beta}-\widehat{{\omega}}{_{\alpha}%
}^{l}(\widehat{Z}_{\overline{\gamma}})\widehat{A}_{\beta l}-\widehat{{\omega}%
}{_{\beta}}^{l}(\widehat{Z}_{\overline{\gamma}})\widehat{A}_{\alpha l})\\
&  =ie^{-f}[(Z_{\overline{\gamma}}+2f_{\overline{\gamma}})\widehat{A}%
_{\alpha\beta}+2(\delta_{\alpha\gamma}\widehat{A}_{\beta l}+\delta
_{\beta\gamma}\widehat{A}_{\alpha l})f^{l}]\\
&  =ie^{-f}(Z_{\overline{\gamma}}+2f_{\overline{\gamma}})e^{-2f}%
({A_{\alpha\beta}+2i}f_{\alpha\beta}-4if_{\alpha}f_{\beta})\\
&  +2e^{-3f}[\delta_{\beta\gamma}(i{A_{\alpha l}-2}f_{\alpha l}+4f_{\alpha
}f_{l})+\delta_{\alpha\gamma}(i{A_{\beta l}-2}f_{\beta l}+4f_{\beta}%
f_{l})]f^{l}\\
&  =e^{-3f}[i{A_{\alpha\beta,\overline{\gamma}}-2f_{\alpha\beta\overline
{\gamma}}+4(f_{\alpha\overline{\gamma}}f_{\beta}+f_{\alpha}f_{\beta
\overline{\gamma}})}]\\
&  +2e^{-3f}[\delta_{\beta\gamma}(i{A_{\alpha l}-2}f_{\alpha l}+4f_{\alpha
}f_{l})+\delta_{\alpha\gamma}(i{A_{\beta l}-2}f_{\beta l}+4f_{\beta}%
f_{l})]f^{l}.
\end{align*}
Contracting the second equation with respect to the Levi metric $\widehat
{h}_{\gamma\overline{\beta}}=h_{\gamma\overline{\beta}}$ yields%
\[%
\begin{array}
[c]{ccc}%
i\widehat{A}_{\alpha\beta},^{\beta} & = & e^{-3f}[i{A_{\alpha\beta},}^{\beta
}-2f_{\alpha\beta}{}^{\beta}+4(f_{\alpha}{}^{\beta}f_{\beta}+f_{\alpha
}f_{\beta}{}^{\beta})\\
&  & \text{ \ \ \ \ \ }{+2(n+1)(iA_{\alpha\beta}-}2{f_{\alpha\beta}%
+4}f_{\alpha}f_{\beta})f^{\beta}].
\end{array}
\]
Thus%
\[%
\begin{array}
[c]{lll}%
\widehat{R}_{\alpha}-in\widehat{A}_{\alpha\beta},^{\beta} & = & e^{-3f}%
[R_{\alpha}-in{A_{\alpha\beta},}^{\beta}-2(n+1)(f_{\beta}{}^{\beta}{}_{\alpha
}+f_{\overline{\beta}}{}^{\overline{\beta}}{}_{\alpha})+2nf_{\alpha\beta}%
{}^{\beta}\\
&  & -2Rf_{\alpha}-2n(n+1)i{A_{\alpha\beta}}f^{\beta}+4(n+1)(f_{\beta}%
{}^{\beta}+f_{\overline{\beta}}{}^{\overline{\beta}})f_{\alpha}\\
&  & -4n(n+1)f^{\beta}{}_{\alpha}f_{\beta}-4n(f_{\alpha}{}^{\beta}f_{\beta
}+f_{\beta}{}^{\beta}f_{\alpha})].
\end{array}
\]
By using the commutation relations (\cite[Lemma 2.3]{l2})
\[
-2(n+1)f_{\beta}{}^{\beta}{}_{\alpha}+2nf_{\alpha\beta}{}^{\beta
}=-2f_{\overline{\beta}}{}^{\overline{\beta}}{}_{\alpha}+2nR_{\alpha
\overline{\beta}}f^{\overline{\beta}}-2in{A_{\alpha\beta}}f^{\beta},
\]
and%
\[%
\begin{array}
[c]{c}%
{f_{\alpha\overline{\beta}}-f_{\overline{\beta}\alpha}=ih_{\alpha
\overline{\beta}}f}_{0},
\end{array}
\]
and by (\ref{5})
\[%
\begin{array}
[c]{c}%
\lbrack(R_{\alpha\overline{\beta}}-\frac{R}{n}h_{\alpha\overline{\beta}%
})-2(n+2)(f_{\alpha\overline{\beta}}-\frac{1}{n}f_{\gamma}{}^{\gamma
}{h_{\alpha\overline{\beta}}})]f^{\overline{\beta}}=e^{f}(\widehat{R}%
_{\alpha\overline{\beta}}-\frac{\widehat{R}}{n}\widehat{h}_{\alpha
\overline{\beta}})f^{\overline{\beta}},
\end{array}
\]
we obtain the following transformation law%
\[%
\begin{array}
[c]{l}%
\widehat{R}_{\alpha}-in\widehat{A}_{\alpha\beta},^{\beta}-2ne^{-2f}%
(\widehat{R}_{\alpha\overline{\beta}}-\frac{\widehat{R}}{n}\widehat{h}%
_{\alpha\overline{\beta}})f^{\overline{\beta}}\\
=e^{-3f}[R_{\alpha}-in{A_{\alpha\beta},}^{\beta}-2(n+2)(f_{\overline{\beta}}%
{}^{\overline{\beta}}{}_{\alpha}+in{A_{\alpha\beta}}f^{\beta})]\\
=e^{-3f}[R_{\alpha}-inA_{\alpha\beta},^{\beta}-2(n+2)P_{\alpha}f].
\end{array}
\]

Then (\ref{4}) follows easily.
\end{proof}

\begin{lemma}
\label{l31} (\cite{l}) Let $(M,J,\theta)$ be a closed, strictly pseudoconvex
CR $(2n+1)$-manifold of $c_{1}(T_{1,0}M)=0$ for $n\geq2$. Then there is a pure
imaginary $1$-form
\[
\sigma=\sigma_{\overline{\alpha}}\theta^{\overline{\alpha}}-\sigma_{\alpha
}\theta^{\alpha}+i\sigma_{0}\theta
\]
with $d\omega_{\alpha}^{\alpha}=d\sigma$ such that
\begin{equation}
\sigma_{\overline{\beta},\overline{\alpha}}=\sigma_{\overline{\alpha
},\overline{\beta}} \label{16a}%
\end{equation}
and
\begin{equation}
\left\{
\begin{array}
[c]{l}%
R_{\alpha\overline{\beta}}=\sigma_{\overline{\beta},\alpha}+\sigma
_{\alpha,\overline{\beta}}-\sigma_{0}h_{\alpha\overline{\beta}},\\
A_{\alpha\beta,}{}^{\beta}=\sigma_{\alpha,0}+i\sigma_{0,\alpha}-A_{\alpha
\beta}\sigma^{\beta}.
\end{array}
\right.  \label{16}%
\end{equation}

\end{lemma}

\begin{lemma}
\label{l32} If $(M,J,\theta)$ is a closed, strictly pseudoconvex CR
$(2n+1)$-manifold of $c_{1}(T_{1,0}M)=0$ for $n\geq2$. Then there exist $u\in
C_{%
\mathbb{R}
}^{\infty}\left(  M\right)  $ and $\gamma=\gamma_{\overline{\alpha}}%
\theta^{\overline{\alpha}}\in\Omega^{0,1}\left(  M\right)  $ such that%
\begin{equation}
W_{\alpha}=2P_{\alpha}u+in\left(  A_{\alpha\beta}\gamma_{\overline{\beta}%
}-\gamma_{\alpha,0}\right)  \label{01}%
\end{equation}
and
\begin{equation}
\gamma_{\overline{\alpha},\overline{\beta}}=\gamma_{\overline{\beta}%
,\overline{\alpha}}\text{ }\mathrm{and}\text{ }\gamma_{\overline{\alpha
},\alpha}=0. \label{15}%
\end{equation}

\end{lemma}

\begin{proof}
By choosing
\[
\eta=\sigma_{\overline{\alpha}}\theta^{\overline{\alpha}},
\]
as in (\ref{2019C}), \textbf{w}here $\sigma$ is chosen from Lemma
\ref{l31},\textbf{\ }then from (\ref{16a})
\[
\overline{\partial}_{b}\eta=0
\]
and there exists%
\[
\varphi=u+iv\in C_{%
\mathbb{C}
}^{\infty}\left(  M\right)
\]

and%
\[
\gamma=\gamma_{\overline{\alpha}}\theta^{\overline{\alpha}}\in\Omega
^{0,1}\left(  M\right)  \cap\ker\left(  \square_{b}\right)
\]

such that%
\begin{equation}
\sigma_{\overline{\alpha}}=\varphi_{\overline{\alpha}}+\gamma_{\overline
{\alpha}}. \label{20}%
\end{equation}

Note that%
\[
\square_{b}\gamma=0\Longrightarrow\overline{\partial}_{b}\gamma=0=\overline
{\partial}_{b}^{\ast}\gamma\Longrightarrow\gamma_{\overline{\alpha}%
,\overline{\beta}}=\gamma_{\overline{\beta},\overline{\alpha}}\text{
}\mathrm{and}\text{ }\gamma_{\overline{\alpha},\alpha}=0
\]

and%
\begin{equation}
\sigma_{\alpha}=\left(  \overline{\varphi}\right)  _{\alpha}+\gamma_{\alpha}.
\label{21}%
\end{equation}

Here $\gamma_{\alpha}=\overline{\gamma_{\overline{\alpha}}}$. From the first
equality in $\left(  \ref{16}\right)  $,%
\begin{equation}
R=\sigma_{\overline{\mu},\mu}+\sigma_{\mu,\overline{\mu}}-n\sigma_{0}.
\label{17}%
\end{equation}

Therefore%
\[%
\begin{array}
[c]{ccl}%
\sigma_{\mu,\overline{\mu}\alpha} & = & (\overline{\varphi})_{,\mu
\overline{\mu}\alpha}+\gamma_{\mu,\overline{\mu}\alpha}\text{ \ }%
\mathrm{by}\text{ }\left(  \ref{21}\right) \\
& = & (\overline{\varphi})_{,\mu\overline{\mu}\alpha}\text{ \ }\mathrm{by}%
\text{ }\left(  \ref{15}\right) \\
& = & (\overline{\varphi})_{,\overline{\mu}\mu\alpha}+in(\overline{\varphi
})_{,0\alpha}\text{ \ }\\
& = & (\overline{\varphi})_{,\overline{\mu}\mu\alpha}+in\left[  (\overline
{\varphi})_{,\alpha0}+A_{\alpha\beta}(\overline{\varphi})_{,\overline{\beta}%
}\right]  \text{ \ }%
\end{array}
\]

and%
\[
\sigma_{\overline{\mu},\mu\alpha}=\varphi_{,\overline{\mu}\mu\alpha}\text{
\ }\mathrm{by}\text{ }\left(  \ref{20}\right)  \text{\quad}\mathrm{and}\text{
}\left(  \ref{15}\right)  .
\]

It follows that
\[%
\begin{array}
[c]{ccl}%
W_{\alpha} & = & \left(  R,_{\alpha}-inA_{\alpha\beta,\overline{\beta}}\right)
\\
& = & \sigma_{\overline{\mu},\mu\alpha}+\sigma_{\mu,\overline{\mu}\alpha
}-in\sigma_{\alpha,0}+inA_{\alpha\beta}\sigma_{\overline{\beta}}\text{
\ }\mathrm{by}\text{ }\left(  \ref{16}\right)  \text{ }\mathrm{and}\text{
}\left(  \ref{17}\right) \\
& = & \varphi_{,\overline{\mu}\mu\alpha}+(\overline{\varphi})_{,\overline{\mu
}\mu\alpha}+inA_{\alpha\beta}(\overline{\varphi})_{,\overline{\beta}}%
-in\gamma_{\alpha,0}+inA_{\alpha\beta}\left(  \varphi_{\overline{\beta}%
}+\gamma_{\overline{\beta}}\right) \\
& = & 2\left(  u_{,\overline{\mu}\mu\alpha}+inA_{\alpha\beta}u_{\overline
{\beta}}\right)  +in\left(  A_{\alpha\beta}\gamma_{\overline{\beta}}%
-\gamma_{\alpha,0}\right) \\
& = & 2P_{\alpha}u+in\left(  A_{\alpha\beta}\gamma_{\overline{\beta}}%
-\gamma_{\alpha,0}\right)  .
\end{array}
.
\]

\end{proof}

We also recall Lemma 6.2 in \cite{l} that states

\begin{lemma}
\label{l33} If $(M,J,\theta)$ is a closed, strictly pseudoconvex CR
$(2n+1)$-manifold of $c_{1}(T_{1,0}M)=0$ for $n\geq2$, then $\widetilde
{\theta}=e^{\frac{2u}{n+2}}\theta$ is a pseudo-Einstein contact form if and
only if%
\begin{equation}
\gamma_{\overline{\alpha},\beta}+\gamma_{\beta,\overline{\alpha}}=0
\label{1988}%
\end{equation}
for all $\alpha,\beta\in I_{n}.$
\end{lemma}

\begin{remark}
Note that the conformal factor $e^{\frac{2u}{n+2}}$is different from Lee's
paper by $\frac{1}{n+2}$ due to the different setting between (\ref{20}) and
\cite[(6.4)]{l}.
\end{remark}

\begin{lemma}
\label{l35} Let $(M,J,\theta)$ be a closed strictly pseudoconvex CR
$(2n+1)$-manifold of $c_{1}(T_{1,0}M)=0$ for $n\geq2.$ If
\[
\int_{M}Ric\left(  \gamma,\gamma\right)  d\mu\geq0,
\]
then $\widetilde{\theta}=e^{\frac{2u}{n+2}}\theta$ is a pseudo-Einstein
contact form and
\begin{equation}
\int_{M}Tor\left(  \gamma,\gamma\right)  d\mu=0, \label{14a}%
\end{equation}
where the smooth function $u\in C_{%
\mathbb{R}
}^{\infty}\left(  M\right)  $ and $\gamma=\gamma_{\overline{\alpha}}%
\theta^{\overline{\alpha}}\in\Omega^{0,1}\left(  M\right)  $ with
$\gamma_{\overline{\alpha},\alpha}=0$ and $\gamma_{\overline{\alpha}%
,\overline{\beta}}=\gamma_{\overline{\beta},\overline{\alpha}}$ are chosen as
in \textbf{Lemma \ref{l32}}.
\end{lemma}

\begin{proof}
It is proved as in \cite{l}
\begin{equation}
\int_{M}Ric\left(  \gamma,\gamma\right)  d\mu+\frac{1}{(n-1)}\underset
{\alpha,\beta}{\sum}\int_{M}\left\vert \gamma_{\alpha,\overline{\beta}%
}\right\vert ^{2}d\mu+\underset{\alpha,\beta}{\sum}\int_{M}\left\vert
\gamma_{\alpha,\beta}\right\vert ^{2}d\mu=0. \label{2018AA}%
\end{equation}
It follows that if the pseudohermitian Ricci curvature is nonnegative
\begin{equation}
\gamma_{\overline{\beta},\alpha}=0=\gamma_{\overline{\beta},\overline{\alpha}}
\label{4d}%
\end{equation}
and by complex conjugate
\begin{equation}
\gamma_{\beta,\overline{\alpha}}=0. \label{4c}%
\end{equation}
Hence by Lemma \ref{l33} that $\widetilde{\theta}=e^{\frac{2u}{n+2}}\theta$ is
a pseudo-Einstein contact form. That is
\begin{equation}
\widetilde{R}_{\alpha\overline{\beta}}=\frac{\widetilde{R}}{n}\widetilde
{h}_{\alpha\overline{\beta}}. \label{4b}%
\end{equation}

On the other hand, \ it follows from (\ref{4}) that
\[
\widetilde{W}_{\alpha}=e^{-\frac{3u}{n+2}}\left[  W_{\alpha}-2\left(
n+2\right)  P_{\alpha}\left(  \frac{u}{n+2}\right)  \right]  +2ne^{-\frac
{2u}{n+2}}\left(  \widetilde{R}_{\alpha\overline{\beta}}-\frac{\widetilde{R}%
}{n}\widetilde{h}_{\alpha\overline{\beta}}\right)  \left(  \frac{u}%
{n+2}\right)  _{,\widetilde{\beta}}%
\]

and then
\begin{equation}
W_{\alpha}=2\left(  n+2\right)  P_{\alpha}\left(  \frac{u}{n+2}\right)
=2P_{\alpha}u. \label{4A}%
\end{equation}

Thus, by \textbf{Lemma \ref{l32}}, we obtain%
\[
\left(  A_{\alpha\beta}\gamma_{\overline{\beta}}-\gamma_{\alpha,0}\right)
=0.
\]

Moreover, from the equality of Lemma 6.3 in \cite{l} i.e.%
\[
\gamma_{\alpha,\overline{\beta}\beta}=i\left(  1-n\right)  \gamma_{\alpha,0}%
\]
and by (\ref{4c})
\begin{equation}
\gamma_{\alpha,0}=0. \label{02}%
\end{equation}
This implies
\[
A_{\alpha\beta}\gamma_{\overline{\beta}}=0.
\]
In particular
\[
\int_{M}Tor\left(  \gamma,\gamma\right)  d\mu=0.
\]

\end{proof}

In this paper, we have another criterion for $\widetilde{\theta}=e^{\frac
{2u}{n+2}}\theta$ to be a pseudo-Einstein contact form.

\begin{lemma}
\label{l34} Let $(M,J,\theta)$ be a closed, strictly pseudoconvex CR
$(2n+1)$-manifold of $c_{1}(T_{1,0}M)=0$ for $n\geq2$. Then $\widetilde
{\theta}=e^{\frac{2u}{n+2}}\theta$ is a pseudo-Einstein contact form if and
only if
\begin{equation}
\left(  A_{\alpha\beta}\gamma_{\overline{\beta}}-\gamma_{\alpha,0}\right)  =0.
\label{2020b}%
\end{equation}

\end{lemma}

\begin{proof}
If $\widetilde{\theta}=e^{\frac{2u}{n+2}}\theta$ is a pseudo-Einstein contact
form, then as the proof of Lemma \ref{l35}, we have
\begin{equation}
\left(  A_{\alpha\beta}\gamma_{\overline{\beta}}-\gamma_{\alpha,0}\right)  =0.
\label{22}%
\end{equation}

Conversely, assume that $\left(  A_{\alpha\beta}\gamma_{\overline{\beta}%
}-\gamma_{\alpha,0}\right)  =0,$ then
\[%
\begin{array}
[c]{ccl}%
0 & = & ni\int_{M}\left(  A_{\alpha\beta}\gamma_{\overline{\beta}}%
-\gamma_{\alpha,0}\right)  \gamma_{\overline{\alpha}}d\mu\\
& = & ni\int_{M}A_{\alpha\beta}\gamma_{\overline{\beta}}\gamma_{\overline
{\alpha}}d\mu-\int_{M}\left(  \gamma_{\alpha,\beta\overline{\beta}}%
-\gamma_{\alpha,\overline{\beta}\beta}-R_{\alpha\overline{\beta}}\gamma
_{\beta}\right)  \gamma_{\overline{\alpha}}d\mu\\
& = & ni\int_{M}A_{\alpha\beta}\gamma_{\overline{\beta}}\gamma_{\overline
{\alpha}}d\mu+\int_{M}Ric\left(  \gamma,\gamma\right)  d\mu-\underset
{\alpha,\beta}{\sum}\int_{M}\left\vert \gamma_{\alpha,\overline{\beta}%
}\right\vert ^{2}d\mu+\underset{\alpha,\beta}{\sum}\int_{M}\left\vert
\gamma_{\alpha,\beta}\right\vert ^{2}d\mu.
\end{array}
\]
Hence
\[
\int_{M}Ric\left(  \gamma,\gamma\right)  d\mu-\frac{n}{2}\int_{M}Tor\left(
\gamma,\gamma\right)  d\mu-\underset{\alpha,\beta}{\sum}\int_{M}\left\vert
\gamma_{\alpha,\overline{\beta}}\right\vert ^{2}d\mu+\underset{\alpha,\beta
}{\sum}\int_{M}\left\vert \gamma_{\alpha,\beta}\right\vert ^{2}d\mu=0.
\]

Again by (\ref{2018AA}), we have
\[
(n-1)\int_{M}Tor\left(  \gamma,\gamma\right)  d\mu+2\underset{\alpha,\beta
}{\sum}\int_{M}\left\vert \gamma_{\alpha,\overline{\beta}}\right\vert ^{2}%
d\mu=0.
\]
On the other hand, it follows from (\ref{2020a}) that
\[%
\begin{array}
[c]{l}%
\underset{\alpha,\beta}{\sum}\int_{M}|\gamma_{\overline{\alpha},\beta}%
+\gamma_{\beta,\overline{\alpha}}|^{2}d\mu\\
=2\underset{\alpha,\beta}{\sum}\int_{M}\left\vert \gamma_{\alpha
,\overline{\beta}}\right\vert ^{2}d\mu+(n-1)\int_{M}Tor(\gamma,\gamma)d\mu.
\end{array}
\]
Hence
\[
\underset{\alpha,\beta}{\sum}\int_{M}|\gamma_{\overline{\alpha},\beta}%
+\gamma_{\beta,\overline{\alpha}}|^{2}d\mu=0.
\]
It follows from (\ref{1988}) that $\widetilde{\theta}=e^{\frac{2u}{n+2}}%
\theta$ is a pseudo-Einstein contact form.
\end{proof}

In particular, if the pseudohermitian is vanishing, it is straightforward to
obtain
\[
\gamma_{\alpha,0}=0.
\]
Therefore, we recapture that $\widetilde{\theta}=e^{\frac{2u}{n+2}}\theta$ is
a pseudo-Einstein contact form as following :

\begin{corollary}
Let $(M,J,\theta)$ be a closed, strictly pseudoconvex CR $(2n+1)$-manifold of
$c_{1}(T_{1,0}M)=0$ and vanishing torsion $A_{\alpha\beta}=0$ for $n\geq2$.
Then $\widetilde{\theta}=e^{\frac{2u}{n+2}}\theta$ is a pseudo-Einstein
contact form.
\end{corollary}

\begin{proof}
Since $\gamma_{\overline{\alpha},\alpha}=0$ and $A_{\alpha\beta}=0,$ by the
commutation relations (\cite{l}) and (\ref{01}),
\[%
\begin{array}
[c]{ccl}%
0 & \leq & n\int_{M}\left\vert \gamma_{\alpha,0}\right\vert ^{2}d\mu\\
& = & n\int_{M}\gamma_{\alpha,0}\gamma_{\overline{\alpha},0}d\mu\\
& = & i\int_{M}\gamma_{\overline{\alpha},0}\left(  R_{,\alpha}-2u_{\overline
{\beta}\beta\alpha}\right)  d\mu\ \\
& = & -i\int_{M}\gamma_{\overline{\alpha}}\left(  R_{,\alpha}-2u_{\overline
{\beta}\beta\alpha}\right)  _{0}d\mu\\
& = & -i\int_{M}\gamma_{\overline{\alpha}}\left(  R_{,0\alpha}-2u_{\overline
{\beta}\beta0\alpha}\right)  d\mu\\
& = & i\int_{M}\gamma_{\overline{\alpha},\alpha}\left(  R_{,0}-2u_{\overline
{\beta}\beta0}\right)  d\mu\\
& = & 0.
\end{array}
\]
Then
\[
\gamma_{\alpha,0}=0
\]
and since $A_{\alpha\beta}=0$
\[
\left(  A_{\alpha\beta}\gamma_{\overline{\beta}}-\gamma_{\alpha,0}\right)
=0.
\]
It follows from (\ref{2020b}) that $\widetilde{\theta}=e^{\frac{2u}{n+2}%
}\theta$ is a pseudo-Einstein contact form.
\end{proof}

Next we come out with the following key Bochner-type formulae for
$\gamma=\gamma_{\overline{\alpha}}\theta^{\overline{\alpha}}.$

\begin{theorem}
\label{t31} Let $(M,J,\theta)$ be a closed, strictly pseudoconvex CR
$(2n+1)$-manifold of $c_{1}(T_{1,0}M)=0$ for $n\geq2.$ Then

(i)%
\begin{equation}
\int_{M}(Ric-\frac{1}{2}Tor)\left(  \gamma,\gamma\right)  d\mu+\underset
{\alpha,\beta}{\sum}\int_{M}\left\vert \gamma_{\alpha,\beta}\right\vert
^{2}d\mu+\frac{1}{2(n-1)}\underset{\alpha,\beta}{\sum}\int_{M}|\gamma
_{\overline{\alpha},\beta}+\gamma_{\beta,\overline{\alpha}}|^{2}d\mu=0.
\label{3}%
\end{equation}

(ii)
\begin{equation}
\frac{n}{2}\int_{M}Tor^{\prime}\left(  \gamma,\gamma\right)  d\mu-\int
_{M}(Q+P_{0}u)ud\mu+\frac{n}{2(n-1)}\underset{\alpha,\beta}{\sum}\int
_{M}|\gamma_{\overline{\alpha},\beta}+\gamma_{\beta,\overline{\alpha}}%
|^{2}d\mu=0. \label{3c}%
\end{equation}

(iii)
\begin{equation}
\int_{M}(Ric-\frac{1}{2}Tor-\frac{1}{2}Tor^{\prime})\left(  \gamma
,\gamma\right)  d\mu+\frac{1}{n}\int_{M}(Q+P_{0}u)ud\mu+\underset{\alpha
,\beta}{\sum}\int_{M}\left\vert \gamma_{\alpha,\beta}\right\vert ^{2}d=0.
\label{3C}%
\end{equation}

Here $Tor\left(  \gamma,\gamma\right)  :=i(A_{\overline{\alpha}\overline
{\beta}}\gamma_{\alpha}\gamma_{\beta}-A_{\alpha\beta}\gamma_{\overline{\alpha
}}\gamma_{\overline{\beta}})$ and $Tor^{\prime}\left(  \gamma,\gamma\right)
:=i(A_{\overline{\alpha}\overline{\beta},\beta}\gamma_{\alpha}-A_{\alpha
\beta,\overline{\beta}}\gamma_{\overline{\alpha}}).$
\end{theorem}

\begin{proof}
From the equality $\left(  \ref{01}\right)  $%
\[
W_{\alpha}=2P_{\alpha}u+in\left(  A_{\alpha\beta}\gamma_{\overline{\beta}%
}-\gamma_{\alpha,0}\right)  ,
\]
we are able to get%
\[%
\begin{array}
[c]{ccl}%
\left(  R,_{\alpha}-inA_{\alpha\beta,\overline{\beta}}\right)  \gamma
_{\overline{\alpha}} & = & W_{\alpha}\gamma_{\overline{\alpha}}\\
& = & 2\left(  u_{\overline{\beta}\beta\alpha}+inA_{\alpha\beta}%
u_{\overline{\beta}}\right)  \gamma_{\overline{\alpha}}+in(A_{\alpha\beta
}\gamma_{\overline{\beta}}-\gamma_{\alpha,0})\gamma_{\overline{\alpha}}\\
& = & 2\left(  u_{\overline{\beta}\beta\alpha}+inA_{\alpha\beta}%
u_{\overline{\beta}}\right)  \gamma_{\overline{\alpha}}+inA_{\alpha\beta
}\gamma_{\overline{\beta}}\gamma_{\overline{\alpha}}-\left(  \gamma
_{\alpha,\beta\overline{\beta}}-\gamma_{\alpha,\overline{\beta}\beta
}-R_{\alpha\overline{\beta}}\gamma_{\beta}\right)  \gamma_{\overline{\alpha}}.
\end{array}
\]

Taking the integration over $M$ of both sides and its conjugation, we have, by
the fact that $\gamma_{\alpha,\overline{\alpha}}=0$,%
\begin{equation}%
\begin{array}
[c]{l}%
\int_{M}\left(  Ric-\frac{n}{2}Tor-\frac{n}{2}Tor^{\prime}\right)  \left(
\gamma,\gamma\right)  d\mu-\underset{\alpha,\beta}{\sum}\int_{M}\left\vert
\gamma_{\alpha,\overline{\beta}}\right\vert ^{2}d\mu\\
+\underset{\alpha,\beta}{\sum}\int_{M}\left\vert \gamma_{\alpha,\beta
}\right\vert ^{2}d\mu-n\int_{M}Tor\left(  d_{b}u,\gamma\right)  d\mu\\
=0.
\end{array}
\label{34}%
\end{equation}
Here $Tor\left(  d_{b}u,\gamma\right)  =i(A_{\overline{\alpha}\overline{\beta
}}u_{\beta}\gamma_{\alpha}-A_{\alpha\beta}u_{\overline{\beta}}\gamma
_{\overline{\alpha}}).$

On the other hand, it follows from equality $\left(  \ref{01}\right)  $ that
\begin{equation}
\left(  R,_{\alpha}-inA_{\alpha\beta,\overline{\beta}}\right)  u_{\overline
{\alpha}}=W_{\alpha}u_{\overline{\alpha}}=\left[  2P_{\alpha}u+in\left(
A_{\alpha\beta}\gamma_{\overline{\beta}}-\gamma_{\alpha,0}\right)  \right]
u_{\overline{\alpha}}. \label{30A}%
\end{equation}
By the fact that $\gamma_{\alpha,\overline{\alpha}}=0$ again, we see that%
\begin{equation}%
\begin{array}
[c]{ccl}%
\int_{M}\gamma_{\alpha,0}u_{\overline{\alpha}}d\mu & = & -\int_{M}%
\gamma_{\alpha}u_{\overline{\alpha}0}d\mu\\
& = & -\int_{M}\gamma_{\alpha}\left(  u_{0\overline{\alpha}}-A_{\overline
{\alpha}\overline{\beta}}u_{\beta}\right)  d\mu\\
& = & \int_{M}A_{\overline{\alpha}\overline{\beta}}u_{\beta}\gamma_{\alpha
}d\mu.
\end{array}
\label{31A}%
\end{equation}
It follows from $\left(  \ref{30A}\right)  \ $and$\ \left(  \ref{31A}\right)
$ that%
\[%
\begin{array}
[c]{l}%
\ \ \ 2\int_{M}Qud\mu+2\int_{M}\left(  P_{0}u\right)  ud\mu\\
=in\int_{M}\left[  \left(  A_{\alpha\beta}u_{\overline{\beta}}\gamma
_{\overline{\alpha}}-A_{\overline{\alpha}\overline{\beta}}u_{\beta}%
\gamma_{\alpha}\right)  -conj\right]  d\mu\\
=-2n\int_{M}Tor\left(  d_{b}u,\gamma\right)  d\mu.
\end{array}
\]
That is
\begin{equation}
\int_{M}Qud\mu+\int_{M}\left(  P_{0}u\right)  ud\mu=-n\int_{M}Tor\left(
d_{b}u,\gamma\right)  d\mu. \label{2020c}%
\end{equation}
Thus by (\ref{34}),%
\begin{equation}%
\begin{array}
[c]{l}%
\int_{M}\left(  Ric-\frac{n}{2}Tor-\frac{n}{2}Tor^{\prime}\right)  \left(
\gamma,\gamma\right)  d\mu-\underset{\alpha,\beta}{\sum}\int_{M}\left\vert
\gamma_{\alpha,\overline{\beta}}\right\vert ^{2}d\mu\\
+\underset{\alpha,\beta}{\sum}\int_{M}\left\vert \gamma_{\alpha,\beta
}\right\vert ^{2}d\mu+\int_{M}(Q+P_{0}u)ud\mu\\
=0.
\end{array}
\label{35}%
\end{equation}

On the other hand, since
\[%
\begin{array}
[c]{l}%
\underset{\alpha,\beta}{\sum}\int_{M}|\gamma_{\overline{\alpha},\beta}%
+\gamma_{\beta,\overline{\alpha}}|^{2}d\mu\\
=2\underset{\alpha,\beta}{\sum}\int_{M}\left\vert \gamma_{\alpha
,\overline{\beta}}\right\vert ^{2}d\mu+(\int_{M}\gamma_{\alpha,\overline
{\beta}}\gamma_{\beta,\overline{\alpha}}d\mu+\mathrm{conj})
\end{array}
\]
and by commutation relations,
\[%
\begin{array}
[c]{l}%
\int_{M}\gamma_{\alpha,\overline{\beta}}\gamma_{\beta,\overline{\alpha}}d\mu\\
=-\int_{M}\gamma_{\alpha}\gamma_{\beta,\overline{\alpha}\overline{\beta}}%
d\mu\\
=i(n-1)\int_{M}A_{\overline{\alpha}\overline{\rho}}\gamma_{\alpha}\gamma
_{\rho}d\mu.
\end{array}
\]
Hence
\begin{equation}%
\begin{array}
[c]{l}%
\underset{\alpha,\beta}{\sum}\int_{M}|\gamma_{\overline{\alpha},\beta}%
+\gamma_{\beta,\overline{\alpha}}|^{2}d\mu\\
=2\underset{\alpha,\beta}{\sum}\int_{M}\left\vert \gamma_{\alpha
,\overline{\beta}}\right\vert ^{2}d\mu+(n-1)\int_{M}Tor(\gamma,\gamma)d\mu.
\end{array}
\label{2020a}%
\end{equation}
This and (\ref{35}) implies%
\begin{equation}%
\begin{array}
[c]{l}%
\int_{M}(Ric-\frac{1}{2}Tor)\left(  \gamma,\gamma\right)  d\mu-\frac{n}{2}%
\int_{M}Tor^{\prime}\left(  \gamma,\gamma\right)  d\mu-\frac{1}{2}%
\underset{\alpha,\beta}{\sum}\int_{M}|\gamma_{\overline{\alpha},\beta}%
+\gamma_{\beta,\overline{\alpha}}|^{2}d\mu\\
+\underset{\alpha,\beta}{\sum}\int_{M}\left\vert \gamma_{\alpha,\beta
}\right\vert ^{2}d\mu+\int_{M}(Q+P_{0}u)ud\mu\\
=0.
\end{array}
\label{2}%
\end{equation}
(\ref{2018AA}) and (\ref{2020a}) implies%
\[
\int_{M}\left(  Ric-\frac{1}{2}Tor\right)  \left(  \gamma,\gamma\right)
d\mu+\frac{1}{2(n-1)}\underset{\alpha,\beta}{\sum}\int_{M}|\gamma
_{\overline{\alpha},\beta}+\gamma_{\beta,\overline{\alpha}}|^{2}d\mu
+\underset{\alpha,\beta}{\sum}\int_{M}\left\vert \gamma_{\alpha,\beta
}\right\vert ^{2}d\mu=0.
\]
By combining (\ref{2}) and (\ref{3}),%
\[
-\frac{n}{2}\int_{M}Tor^{\prime}\left(  \gamma,\gamma\right)  d\mu-\frac
{n}{2(n-1)}\underset{\alpha,\beta}{\sum}\int_{M}|\gamma_{\overline{\alpha
},\beta}+\gamma_{\beta,\overline{\alpha}}|^{2}d\mu+\int_{M}(Q+P_{0}u)ud\mu=0.
\]
By combining (\ref{35}) and (\ref{2018AA}),%
\[
\int_{M}(Ric-\frac{1}{2}Tor-\frac{1}{2}Tor^{\prime})\left(  \gamma
,\gamma\right)  d\mu+\frac{1}{n}\int_{M}(Q+P_{0}u)ud\mu+\underset{\alpha
,\beta}{\sum}\int_{M}\left\vert \gamma_{\alpha,\beta}\right\vert ^{2}d\mu=0.
\]

\end{proof}

\section{Pseudo-Einstein Contact Structures}

Now, with the help of the lemmas in the last section, we are able to give the
existence theorems for pseudo-Einstein contact structures as in Theorem
\ref{T2} and Theorem \ref{T2a}.

\begin{lemma}
\label{l41} Let $(M,J,\theta)$ be a closed, strictly pseudoconvex CR
$(2n+1)$-manifold of $c_{1}(T_{1,0}M)=0$ for $n\geq2.$ Then

(i)
\begin{equation}
Q_{\ker}=0. \label{2019b}%
\end{equation}

(ii)
\begin{equation}
Q^{\perp}+P_{0}u^{\perp}=0, \label{2019a}%
\end{equation}
if $\widetilde{\theta}=e^{\frac{2u}{n+2}}\theta$ is a pseudo-Einstein contact
form. Here $Q=Q_{\ker}+Q^{\perp}$. $Q^{\perp}$ is in $(\ker P_{0})^{\perp}$
which is perpendicular to the kernel of self-adjoint Paneitz operator $P_{0}$
in the $L^{2}$ norm with respect to the volume form $d\mu$ $=$ $\theta\wedge
d\theta$.
\end{lemma}

\begin{proof}
(i) We observe that the equality (\ref{01}) still holds if we replace $u$ by
$(u+CQ_{\ker}).$ It follows from the Bochner-type formula (\ref{2}) that
\[%
\begin{array}
[c]{l}%
\int_{M}(Ric-\frac{1}{2}Tor)\left(  \gamma,\gamma\right)  d\mu-\frac{n}{2}%
\int_{M}Tor^{\prime}\left(  \gamma,\gamma\right)  d\mu-\frac{1}{2}%
\underset{\alpha,\beta}{\sum}\int_{M}|\gamma_{\overline{\alpha},\beta}%
+\gamma_{\beta,\overline{\alpha}}|^{2}d\mu\\
+\underset{\alpha,\beta}{\sum}\int_{M}\left\vert \gamma_{\alpha,\beta
}\right\vert ^{2}d\mu+\int_{M}\left(  P_{0}u\right)  ud\mu+\int_{M}%
Qud\mu+C\int_{M}(Q_{\ker})^{2}d\mu\\
=0.
\end{array}
\]
However, if $\int_{M}(Q_{\ker})^{2}d\mu$ is not zero, this will lead to a
contradiction by choosing the constant $C<<-1$ or $C>>1.$ Then we are done.

(ii) If $\widetilde{\theta}=e^{\frac{2u}{n+2}}\theta$ is a pseudo-Einstein
contact form, it follows from Lemma \ref{l34} that
\[
\left(  A_{\alpha\beta}\gamma_{\overline{\beta}}-\gamma_{\alpha,0}\right)
=0.
\]
Then from Lemma \ref{l32}
\[
W_{\alpha}=2P_{\alpha}u.
\]
Hence
\[
(W_{\alpha})_{\overline{\alpha}}=2(P_{\alpha}u)_{\overline{\alpha}}.
\]
Taking its conjugacy in both sides
\[
-Q=P_{0}u
\]
and then from (\ref{2019b})
\[
Q^{\perp}+P_{0}u^{\perp}=0.
\]

\end{proof}

We observe that the CR $Q$-curvature is vanishing when it is pseudo-Einstein.
On the other hand, it is unknown whether there is any obstruction to the
existence of a contact form $\theta$ of vanishing CR $Q$-curvature
(\cite{ccc}, \cite{cks}). Our first goal is to justify the case whether a
contact form $\theta$ is pseudo-Einstein whenever its CR $Q$-curvature is CR
plurihramonic consisting of infinite dimensional kernel of the CR Paneitz
operator $P_{0}$ in a closed strictly pseudoconvex CR $(2n+1)$-manifold
$(M,J,\theta)$ for $n\geq2$. The following proposition is due to
(\ref{2018AA}) and Lemma \ref{l35} that

\begin{proposition}
\label{P1} (\cite{l}) Let $(M,J,\theta)$ be a closed, strictly pseudoconvex CR
$(2n+1)$-manifold of $c_{1}(T_{1,0}M)=0,$ $n\geq2$. Suppose that
\begin{equation}
\int_{M}Ric\left(  \gamma,\gamma\right)  d\mu\geq0. \label{2019AA}%
\end{equation}
Then

(i) $\widetilde{\theta}=e^{\frac{2u}{n+2}}\theta$ is a pseudo-Einstein contact form.

(ii) $\theta$ is also a pseudo-Einstein contact form if the CR\ $Q$-curvature
of $\theta$ is CR-pluriharmonic $(i.e.$ $Q^{\bot}=0)$.
\end{proposition}

In general, we hope to replace the nonnegative assumption (\ref{2019AA}) by
more natural pseudohermitian curvatures (\ref{2019A}) which is a combination
of pseudohermitian Ricci curvature and torsion. In fact, the CR analogue of
Bochner formula states that
\begin{equation}%
\begin{array}
[c]{ccl}%
\frac{1}{2}\Delta_{b}|\nabla_{b}u|^{2} & = & |(\nabla^{H})^{2}u|^{2}%
+(1+\frac{2}{n})<\nabla_{b}u,\nabla_{b}\Delta_{b}u>_{L_{\theta}}\\
&  & +[2Ric-(n+2)Tor]((\nabla_{b}u)_{\mathbf{C}},(\nabla_{b}u)_{\mathbf{C}})\\
&  & -\frac{4}{n}<Pu+\overline{P}u,d_{b}u>_{L_{\theta}^{\ast}}.
\end{array}
\label{2020ABC}%
\end{equation}
Here $(\nabla_{b}u)_{\mathbf{C}}=u_{\bar{\alpha}}Z_{\alpha}$ is the
corresponding complex $(1,0)$-vector field of $\ \nabla_{b}u$ and
$d_{b}u=u_{\alpha}\theta^{\alpha}+u_{\overline{\alpha}}\theta^{\overline
{\alpha}}.$ We refer\ this pseudohermitian curvature quantity to our previous
results as in \cite{cc}.

More precisely, it follows from Lemma \ref{l41} and the CR Bochner-type
formulae (\ref{2}), (\ref{3}), one can derive the following :

\begin{theorem}
\label{T2} Let $(M,J,\theta)$ be a closed, strictly pseudoconvex CR
$(2n+1)$-manifold of $c_{1}(T_{1,0}M)=0$ for $n\geq2$. Assume that
\begin{equation}
\int_{M}(Ric-\frac{1}{2}Tor)\left(  \gamma,\gamma\right)  d\mu\geq0.
\label{2019A}%
\end{equation}
Then

(i) $\widetilde{\theta}=e^{\frac{2u}{n+2}}\theta$ is a pseudo-Einstein contact form.

(ii) $\theta$ is also a pseudo-Einstein contact form if the CR\ $Q$-curvature
of $\theta$ is CR-pluriharmonic $(i.e.$ $Q^{\bot}=0)$.
\end{theorem}

\begin{proof}
It follows from (\ref{3}) that
\[
\gamma_{\overline{\alpha},\beta}+\gamma_{\beta,\overline{\alpha}}=0.
\]
Hence, by Lemma \ref{l33}, $\widetilde{\theta}=e^{\frac{2u}{n+2}}\theta$ is a
pseudo-Einstein contact form. On the other hand, if the CR\ $Q$-curvature is
CR-pluriharmonic (i.e. $Q^{\bot}=0),$ then by (\ref{2019a}) and (\ref{2019b}%
),
\[
u^{\bot}=0
\]
for $u=u_{\ker}+u^{\perp}.$ Thus by (\ref{4A}),
\[
W_{\alpha}=0.
\]
Then $\theta$ is also a pseudo-Einstein contact form.
\end{proof}

\begin{corollary}
\label{c41} Let $(M,J,\theta)$ be a closed, strictly pseudoconvex CR
$(2n+1)$-manifold of $c_{1}(T_{1,0}M)=0$ for $n\geq2$. Assume that
\[
\int_{M}(Ric-\frac{1}{2}Tor)\left(  \gamma,\gamma\right)  d\mu\geq0.
\]
Then
\begin{equation}
\int_{M}Tor^{\prime}\left(  \gamma,\gamma\right)  d\mu=0. \label{2019c}%
\end{equation}
Here $Tor^{\prime}\left(  \gamma,\gamma\right)  :=i(A_{\overline{\alpha
}\overline{\beta},\beta}\gamma_{\alpha}-A_{\alpha\beta,\overline{\beta}}%
\gamma_{\overline{\alpha}})=2\operatorname{Re}(i(A_{\overline{\alpha}%
\overline{\beta},\beta}\gamma_{\alpha}).$
\end{corollary}

\begin{proof}
It follows from (\ref{3}) and the assumption that
\[
\int_{M}(Ric-\frac{1}{2}Tor)\left(  \gamma,\gamma\right)  d\mu=0
\]
and
\[
0=\underset{\alpha,\beta}{\sum}\int_{M}|\gamma_{\overline{\alpha},\beta
}+\gamma_{\beta,\overline{\alpha}}|^{2}d\mu=\underset{\alpha,\beta}{\sum}%
\int_{M}\left\vert \gamma_{\alpha,\beta}\right\vert ^{2}d\mu.
\]
Hence by (\ref{2}), we have
\[
\int_{M}Qud\mu+\int_{M}\left(  P_{0}u\right)  ud\mu-\frac{n}{2}\int
_{M}Tor^{\prime}\left(  \gamma,\gamma\right)  d\mu=0.
\]
Finally, it follows from (\ref{2019a}) that
\[
\int_{M}Qud\mu+\int_{M}\left(  P_{0}u\right)  ud\mu=0
\]
and then
\[
\int_{M}Tor^{\prime}\left(  \gamma,\gamma\right)  d\mu=0.
\]

\end{proof}

Now we want to relate the existence of pseudo-Einstein contact forms with the
first Kohn-Rossi cohomology group $H_{\overline{\partial}_{b}}^{0,1}(M)$. By
combining the Bochner formulae (\ref{3c}), we have

\begin{theorem}
\label{T2a} Let $(M,J,\theta)$ be a closed, strictly pseudoconvex CR
$(2n+1)$-manifold of $c_{1}(T_{1,0}M)=0$ with
\[
d\omega_{\alpha}^{\alpha}=d\sigma
\]
for $\sigma=\sigma_{\overline{\alpha}}\theta^{\overline{\alpha}}%
-\sigma_{\alpha}\theta^{\alpha}+i\sigma_{0}\theta.$ Assume that
\begin{equation}
\eta=\sigma_{\overline{\alpha}}\theta^{\overline{\alpha}}\in\ker\left(
\square_{b}\right)  . \label{20b}%
\end{equation}
Then $\widetilde{\theta}$ is pseudo-Einstein if and only if
\begin{equation}
\int_{M}Tor^{\prime}\left(  \eta,\eta\right)  d\mu=0. \label{3CC}%
\end{equation}
In fact, $\theta$ is also pseudo-Einstein.
\end{theorem}

\begin{remark}
We observe that $\eta=\sigma_{\overline{\alpha}}\theta^{\overline{\alpha}}$ is
a smooth representative of the first Kohn-Rossi cohomology group
$H_{\overline{\partial}_{b}}^{0,1}(M)$ \ if and only if
\[
\sigma_{\alpha,\overline{\alpha}}=0\text{ \ \ \textrm{\ and }\ \ }%
\sigma_{\alpha,\beta}=\sigma_{\beta,\alpha}.
\]
However, $\sigma_{\alpha,\beta}=\sigma_{\beta,\alpha}$ holds if $d\omega
_{\alpha}^{\alpha}=d\sigma.$ If $\sigma_{\alpha,\beta}=0,$ then
\[
\int_{M}Tor^{\prime}\left(  \eta,\eta\right)  d\mu=i\int_{M}(A_{\alpha\beta
}\sigma_{\overline{\alpha},\overline{\beta}}-A_{\overline{\alpha}%
\overline{\beta}}\sigma_{\alpha,\beta})d\mu=0.
\]

\end{remark}

\begin{proof}
It follows from (\ref{20a}), (\ref{20aa}) and (\ref{20b}) that
\[
\sigma_{\overline{\alpha}}=\gamma_{\overline{\alpha}}%
\]
and
\[
u^{\bot}=0.
\]
Here we use the fact that the Kohn-Rossi cohomology group $H_{\overline
{\partial}_{b}}^{0,1}(M)$ has a unique smooth representative $\gamma\in
\ker\left(  \square_{b}\right)  .$ This implies
\[
\int_{M}(Q+P_{0}u)ud\mu=\int_{M}(Q^{\bot}+P_{0}u^{\bot})u^{\bot}d\mu=0.
\]
It follows from Bochner formula (\ref{3c}) that
\begin{equation}
\frac{n}{2}\int_{M}Tor^{\prime}\left(  \gamma,\gamma\right)  d\mu+\frac
{n}{2(n-1)}\underset{\alpha,\beta}{\sum}\int_{M}|\gamma_{\overline{\alpha
},\beta}+\gamma_{\beta,\overline{\alpha}}|^{2}d\mu=0. \label{3cccc}%
\end{equation}
Then
\[
\int_{M}Tor^{\prime}\left(  \eta,\eta\right)  d\mu=0
\]
if and only if
\[
\int_{M}\underset{\alpha,\beta}{\sum}|\gamma_{\overline{\alpha},\beta}%
+\gamma_{\beta,\overline{\alpha}}|^{2}d\mu=0.
\]
That is
\[
\gamma_{\overline{\alpha},\beta}+\gamma_{\beta,\overline{\alpha}}=0.
\]

All these imply that $\widetilde{\theta}=e^{\frac{2u}{n+2}}\theta$ is a
pseudo-Einstein contact form as well as $\theta$ due to $u^{\bot}=0.$
\end{proof}

We observe that if the first Kohn-Rossi cohomology group $H_{\overline
{\partial}_{b}}^{0,1}(M)$ is vanishing, it follows from Lemma \ref{l33} that
$\widetilde{\theta}=e^{\frac{2u}{n+2}}\theta$ is a pseudo-Einstein contact
form. As a consequence of Theorem \ref{T2a}, we have

\begin{corollary}
\label{C12} Let $(M,J,\theta)$ be a closed, strictly pseudoconvex CR
$(2n+1)$-manifold of $c_{1}(T_{1,0}M)=0,$ $n\geq2$ with $d\omega_{\alpha
}^{\alpha}=d\sigma$ for some $\sigma=\sigma_{\overline{\alpha}}\theta
^{\overline{\alpha}}-\sigma_{\alpha}\theta^{\alpha}+i\sigma_{0}\theta.$ Assume
that either

(i) the first Kohn-Rossi cohomology group $H_{\overline{\partial}_{b}}%
^{0,1}(M)$ is vanishing or

(ii)
\begin{equation}
d\sigma=\Theta=id(f\theta) \label{2020EEE}%
\end{equation}
for some smooth, real-valued function $f$. Then $\widetilde{\theta}$ is the
pseudo-Einstein contact form.
\end{corollary}

\section{The CR Analogue of Frankel Conjecture}

We affirm the CR analogue of Frankel conjecture in a closed, spherical,
strictly pseudoconvex CR $(2n+1)$-manifold.

\begin{lemma}
\label{l51} Let $(M,J,\theta)$ be a closed, spherical, strictly pseudoconvex
CR $(2n+1)$-manifold with the pseudo-Eisntein contact form $\theta$ for
$n\geq2$. Then
\[%
\begin{array}
[c]{ccl}%
0 & = & \frac{n+2}{n+1}\int_{M}k%
{\textstyle\sum_{\alpha,\gamma}}
|A_{\alpha\gamma}|^{2}d\mu+\int_{M}%
{\textstyle\sum_{\alpha,\gamma,\sigma}}
|A_{\alpha\gamma,\sigma}|^{2}d\mu\\
&  & +\frac{1}{n-1}[\int_{M}%
{\textstyle\sum_{\alpha,\gamma,\beta}}
|A_{\alpha\gamma,\overline{\beta}}|^{2}d\mu-n\int_{M}%
{\textstyle\sum_{\alpha}}
A_{\overline{\alpha}\overline{\beta},\beta}A_{\alpha\gamma,\overline{\gamma}%
}d\mu].
\end{array}
\]
Here $k:=\frac{R}{n}.$
\end{lemma}

\begin{proof}
Since $\theta$\ is pseudo-Einstein, it follows that
\begin{equation}%
\begin{array}
[c]{c}%
R_{\alpha\overline{\beta}}=\frac{R}{n}h_{\alpha\overline{\beta}}%
:=kh_{\alpha\overline{\beta}}.
\end{array}
\label{18}%
\end{equation}
Here $k:=$ $\frac{R}{n}.$ Since $J$ is spherical, it follows from (\ref{d12})
and (\ref{18}) that
\begin{equation}%
\begin{array}
[c]{ccl}%
R_{\beta\overline{\alpha}\lambda\overline{\sigma}} & = & \frac{k}%
{n+2}[h_{\beta\overline{\alpha}}h_{\lambda\overline{\sigma}}+h_{\lambda
\overline{\alpha}}h_{\beta\overline{\sigma}}+\delta_{\beta}^{\alpha}%
h_{\lambda\overline{\sigma}}+\delta_{\lambda}^{\alpha}h_{\beta\overline
{\sigma}}]\\
&  & -\frac{nk}{(n+1)(n+2)}[\delta_{\beta}^{\alpha}h_{\lambda\overline{\sigma
}}+\delta_{\lambda}^{\alpha}h_{\beta\overline{\sigma}}]\\
& = & \frac{k}{n+2}[h_{\beta\overline{\alpha}}h_{\lambda\overline{\sigma}%
}+h_{\lambda\overline{\alpha}}h_{\beta\overline{\sigma}}]\\
&  & +\frac{k}{(n+1)(n+2)}[\delta_{\beta}^{\alpha}h_{\lambda\overline{\sigma}%
}+\delta_{\lambda}^{\alpha}h_{\beta\overline{\sigma}}].
\end{array}
\label{18a}%
\end{equation}
Again by \cite[(2.15)]{l},%
\[
A_{\alpha\rho,\beta\overline{\gamma}}=ih_{\beta\overline{\gamma}}A_{\alpha
\rho,0}+R_{\alpha}{}^{\kappa}{}_{\beta\overline{\gamma}}A_{\kappa\rho}%
+R_{\rho}{}^{\kappa}{}_{\beta\overline{\gamma}}A_{\alpha\kappa}+A_{\alpha
\rho,\overline{\gamma}\beta}.
\]
Contracting both sides by $h^{\beta\overline{\gamma}}$%
\[%
\begin{array}
[c]{ccl}%
A_{\alpha\rho,\gamma}{}^{\gamma} & = & inA_{\alpha\rho,0}+R_{\alpha}{}%
^{\kappa}{}_{\gamma}{}^{\gamma}A_{\kappa\rho}+R_{\rho}{}^{\kappa}{}_{\gamma}%
{}^{\gamma}A_{\alpha\kappa}+A_{\alpha\rho,\overline{\gamma}}{}^{\overline
{\gamma}}\\
& = & inA_{\alpha\rho,0}+R_{\alpha\overline{\kappa}}A^{\overline{\kappa}}%
{}_{\rho}+R_{\rho\overline{\kappa}}A^{\overline{\kappa}}{}_{\alpha}%
+A_{\alpha\rho,\overline{\gamma}}{}^{\overline{\gamma}}\\
& = & inA_{\alpha\rho,0}+kh_{\alpha\overline{\kappa}}A^{\overline{\kappa}}%
{}_{\rho}+kh_{\rho\overline{\kappa}}A^{\overline{\kappa}}{}_{\alpha}%
+A_{\alpha\rho,\overline{\gamma}}{}^{\overline{\gamma}}\\
& = & inA_{\alpha\rho,0}+2kA_{\alpha\rho}+A_{\alpha\rho,\overline{\gamma}}%
{}^{\overline{\gamma}}.
\end{array}
\]
That is
\begin{equation}
A_{\alpha\gamma,\sigma}{}^{\sigma}=inA_{\alpha\gamma,0}+2kA_{\alpha\gamma
}+A_{\alpha\gamma,\overline{\sigma}}{}^{\overline{\sigma}} \label{18c}%
\end{equation}
for all $\alpha,\gamma.$ Next we claim that
\begin{equation}%
\begin{array}
[c]{c}%
inA_{\alpha\gamma,0}=-\frac{nk}{n+1}A_{\alpha\gamma}+\frac{n}{n-1}%
(A_{\alpha\beta,\overline{\beta}\gamma}-A_{\alpha\gamma,\overline{\beta}\beta
}).
\end{array}
\label{18d}%
\end{equation}
Again from \cite[(2.9)]{l},
\[%
\begin{array}
[c]{l}%
A_{\alpha\rho,\overline{\beta}\gamma}-A_{\alpha\gamma,\overline{\beta}\rho
}=ih_{\rho\overline{\beta}}A_{\alpha\gamma,0}-ih_{\gamma\overline{\beta}%
}A_{\alpha\rho,0}+R_{\alpha\overline{\beta}\rho\overline{\sigma}}%
A^{\overline{\sigma}}{}_{\gamma}-R_{\alpha\overline{\beta}\gamma
\overline{\sigma}}A^{\overline{\sigma}}{}_{\rho}%
\end{array}
\]
Contracting both sides by $h^{\rho\overline{\beta}},$%
\[
inA_{\alpha\gamma,0}-i\delta_{\gamma}^{\rho}A_{\alpha\rho,0}+R_{\alpha}%
{}^{\rho}{}_{\rho\overline{\sigma}}A^{\overline{\sigma}}{}_{\gamma}-R_{\alpha
}{}^{\rho}{}_{\gamma\overline{\sigma}}A^{\overline{\sigma}}{}_{\rho}%
=A_{\alpha\beta,\overline{\beta}\gamma}-A_{\alpha\gamma,\overline{\beta}\beta
}.
\]
Hence%
\[
i(n-1)A_{\alpha\gamma,0}+R_{\alpha\overline{\sigma}}A^{\overline{\sigma}}%
{}_{\gamma}-R_{\alpha}{}^{\rho}{}_{\gamma\overline{\sigma}}A^{\overline
{\sigma}}{}_{\rho}=A_{\alpha\beta,\overline{\beta}\gamma}-A_{\alpha
\gamma,\overline{\beta}\beta}%
\]
and thus
\[
i(n-1)A_{\alpha\gamma,0}+kA_{\alpha\gamma}-R_{\alpha}{}^{\rho}{}%
_{\gamma\overline{\sigma}}A^{\overline{\sigma}}{}_{\rho}=A_{\alpha
\beta,\overline{\beta}\gamma}-A_{\alpha\gamma,\overline{\beta}\beta}.
\]
On the other hand,%
\[%
\begin{array}
[c]{ccl}%
R_{\alpha}{}^{\rho}{}_{\gamma\overline{\sigma}}A^{\overline{\sigma}}{}_{\rho}
& = & \frac{k}{n+2}[h_{\alpha\overline{\rho}}h_{\gamma\overline{\sigma}%
}+h_{\gamma\overline{\rho}}h_{\alpha\overline{\sigma}}]A^{\overline{\sigma}}%
{}_{\rho}\\
&  & +\frac{k}{(n+1)(n+2)}[\delta_{\alpha}^{\rho}h_{\gamma\overline{\sigma}%
}+\delta_{\gamma}^{\rho}h_{\alpha\overline{\sigma}}]A^{\overline{\sigma}}%
{}_{\rho}\\
& = & \frac{2k}{n+1}A_{\alpha\gamma}.
\end{array}
\]
All these imply%
\[%
\begin{array}
[c]{c}%
i(n-1)A_{\alpha\gamma,0}+\frac{n-1}{n+1}kA_{\alpha\gamma}=A_{\alpha
\beta,\overline{\beta}\gamma}-A_{\alpha\gamma,\overline{\beta}\beta}%
\end{array}
\]
for $n\geq2$. Thus (\ref{18d}) follows. Next, from (\ref{18c}) and
(\ref{18d}), we obtain%
\[%
\begin{array}
[c]{ccl}%
A_{\alpha\gamma,\sigma}{}^{\sigma} & = & inA_{\alpha\gamma,0}+2kA_{\alpha
\gamma}+A_{\alpha\gamma,\overline{\sigma}}{}^{\overline{\sigma}}\\
& = & \frac{n+2}{n+1}kA_{\alpha\gamma}+\frac{n}{n-1}(A_{\alpha\beta
,\overline{\beta}\gamma}-A_{\alpha\gamma,\overline{\beta}\beta})+A_{\alpha
\gamma,\overline{\sigma}}{}^{\overline{\sigma}}.
\end{array}
\]
We integrate both sides with $A^{\alpha\gamma}$ to get%
\begin{equation}%
\begin{array}
[c]{ccl}%
0 & = & \frac{n+2}{n+1}\int_{M}k%
{\textstyle\sum_{\alpha,\gamma}}
|A_{\alpha\gamma}|^{2}d\mu+\int_{M}%
{\textstyle\sum_{\alpha,\gamma,\sigma}}
|A_{\alpha\gamma,\sigma}|^{2}d\mu\\
&  & +\frac{1}{n-1}[\int_{M}%
{\textstyle\sum_{\alpha,\gamma,\beta}}
|A_{\alpha\gamma,\overline{\beta}}|^{2}d\mu-n\int_{M}%
{\textstyle\sum_{\alpha}}
A_{\overline{\alpha}\overline{\beta},\beta}A_{\alpha\gamma,\overline{\gamma}%
}d\mu].
\end{array}
\label{2020}%
\end{equation}

\end{proof}

\begin{theorem}
\label{T51} Let $(M,J,\theta)$ be a closed, spherical, strictly pseudoconvex
CR $(2n+1)$-manifold with pseuodo-Einstein contact form $\theta$ of positive
constant Tanaka-Webster scalar curvature. Then the universal covering of $M$
must be globally CR equivalent to a standard CR sphere.
\end{theorem}

\begin{proof}
Since\textbf{ }%
\[
R_{\alpha\overline{\beta}}{}_{,\beta}=R_{\alpha}-i(n-1)A_{\alpha\beta}%
{}_{,\overline{\beta}},
\]
if $R_{\alpha\overline{\beta}}=\frac{R}{n}h_{\alpha\overline{\beta}}$ and $R$
is constant, then
\[
A_{\alpha\gamma,\overline{\gamma}}=0.
\]
It follows from Lemma \ref{l51} that if $k>0$
\[
\frac{n+2}{n+1}k\int_{M}%
{\textstyle\sum_{\alpha,\gamma}}
|A_{\alpha\gamma}|^{2}d\mu+\int_{M}%
{\textstyle\sum_{\alpha,\gamma,\sigma}}
|A_{\alpha\gamma,\sigma}|^{2}d\mu+\frac{1}{n-1}\int_{M}%
{\textstyle\sum_{\alpha,\gamma,\beta}}
|A_{\alpha\gamma,\overline{\beta}}|^{2}d\mu=0
\]
and
\[
A_{\alpha\gamma}=0.
\]
Moreover, it follows from (\ref{18a}) that
\[
R_{\beta\overline{\alpha}\lambda\overline{\sigma}}=\frac{R}{n(n+1)}%
[h_{\beta\overline{\alpha}}h_{\lambda\overline{\sigma}}+h_{\lambda
\overline{\alpha}}h_{\beta\overline{\sigma}}].
\]
Hence $(M,\theta)$ is a closed, Sasakian CR $(2n+1)$-manifold of positive
constant pseudohermitian bisectional curvature. Hence manifolds always admit
Riemannian metrics with positive Ricci curvature (\cite{cc}), so they must
have finite fundamental group. It follows from (\cite{t}) that the universal
covering of $M$ is CR equivalent to a CR standard Sphere $\mathbf{S}^{2n+1}$
in $\mathbb{C}^{n+1}.$
\end{proof}

Then \textbf{the proofs of Theorem \ref{T13} and Theorem \ref{T14} are
completed.}


\begin{thebibliography}{9999}                                                                                             %


\bibitem[BS]{bs}D. Burns, Jr. and S. Shnider, Spherical Hypersurfaces in
Complex Manifolds, Inventiones math. 33 (1976), 223- 246.

\bibitem[CC]{cc}S.-C. Chang and H.-L. Chiu, Nonnegativity of CR Paneitz
operator and its Application to the CR Obata's Theorem in a Pseudohermitian
$(2n+1)$-Manifold, Journal of Geometric Analysis, Vol. 19 (2009), 261-287.

\bibitem[CCC]{ccc}S.-C. Chang, J.-H. Cheng and H.-L. Chiu, The Fourth-order
Q-curvature flow on a CR 3-manifold, Indiana Univ. Math. J., Vol. 56, No. 4
(2007), 1793-1826.

\bibitem[CKL]{ckl}S.-C. Chang, T.-J. Kuo, C. Lin, Pseudo-Einstein structure,
eigenvalue estimate for the CR Paneitz operator and its applications to
uniformization theorem, arXiv:1807.08898 [math.DG].

\bibitem[CKS]{cks}S.-C. Chang, T.-J. Kuo and T. Saotome, On existence of the
vanishing CR $Q$-curvature in a closed CR $3$-manifold, preprint.

\bibitem[CJ]{cj}S.-S. Chern and S.-Y. Ji, On the Riemann mapping theorem,
Annals of Math., Vol 144 (1996), 421-439.

\bibitem[DT]{dt}S. Dragomir and G. Tomassini, Differential Geometry and
Analysis on CR manifolds, Progress in Mathematics, Volume 246, Birkhauser 2006.

\bibitem[F]{f}T. Frankel, Manifolds with positive curvature, Pacific J. Math.
11 (1961), 165-174.

\bibitem[FH]{fh}C. Fefferman and K. Hirachi, Ambient Metric Construction of
$Q$-Curvature in Conformal and $CR$ Geometries, Math. Res. Lett., 10, No. 5-6
(2003), 819-831.

\bibitem[Fo]{fo}G. B. Folland, Subelliptic Estimates and Function Spaces on
Nilpotent Lie Groups, Arkiv for Mat. 13 (1975), 161-207.

\bibitem[FS]{fs}G. B. Folland and E. M. Stein, Estimates for the
$\bar{\partial}_{b}$ Complex and Analysis on the Heisenberg Group, Comm. Pure
Appl. Math., 27 (1974), 429-522.

\bibitem[GG]{gg}A. R. Gover and C. R. Graham, $CR$ Invariant Powers of the
Sub-Laplacian, J. Reine Angew. Math. 583 (2005), 1-27.

\bibitem[GL]{gl}C. R. Graham and J. M. Lee, Smooth Solutions of Degenerate
Laplacians on Strictly Pseudoconvex Domains, Duke Math. J., 57 (1988), 697-720.

\bibitem[Gr]{gr}A. Greenleaf: The first eigenvalue of a Sublaplacian on a
Pseudohermitian manifold. \textit{Comm. Part. Diff. Equ. } \textbf{10(2)}
(1985), no.3 191--217.

\bibitem[H]{h}K. Hirachi, Scalar pseudohermitian invariants and the Szego
kernel on three-dimensional CR manifolds, Complex Geometry, Lect. Notes in
Pure and Appl. Math. 143, 67-76, Dekker (1993).

\bibitem[HL]{hl}R. Harvey and B. Lawson, On boundaries of complex analytic
varieties I, Ann. of Math. 102 (1975), 233--290.

\bibitem[HS]{hs}W. He and S. Sun, Frankel conjecture and Sasaki geometry,
Advances in Mathematics, 291 (2016), 912--960.

\bibitem[Lee]{l}J.M. Lee, Pseudo-Einstein Structures on CR manifolds, Amer. J.
Math. 110 (1988), 157-178.

\bibitem[Lee2]{l2}J. M. Lee, The Fefferman Metric and Pseudohermitian
Invariants, Trans. A.M.S., 296 (1986), 411-429.

\bibitem[K]{k}J.J. Kohn, Boundaries of Complex Manifolds, Proc. Conf. on
Complex Analysis, Minneapolis,1964, Springer-Verlag, 81--94 (1965).

\bibitem[KT]{kt}Y. Kamishima and T. Tsuboi, CR-structures on Seifert
manifolds, Invent. math. 104 (1991), 149-163.

\bibitem[M]{m}S. Mori, Projective Manifolds with Ample Tangent Bundles. Ann.
of Math. 110 (1979), 593-606.

\bibitem[SY]{sy}Yum-Tong Siu and Shing-Tung Yau, Compact Kaehler Manifolds of
Positive Bisectional Curvature, Inventiones math. 59 (1980), 189-204.

\bibitem[T]{t}S. Tanno, Sasakian manifolds with constant $\phi$-holomorphic
sectional curvature, T\^{o}hoko Math. Journ. \textbf{21} (1969), 501-507.

\bibitem[Ta]{ta}N. Tanaka, \textit{A Differential Geometric Study on Strongly
Pseudoconvex Manifolds, }1975, Kinokuniya Co. Ltd., Tokyo.

\bibitem[We]{we}S. M. Webster, \textit{Pseudohermitian structures on a real
hypersurface, }J. Diff. Geom. 13 (1978), 25-41.

\bibitem[Y]{y}S.S.-T. Yau, Kohn--Rossi cohomology and its application to the
complex Plateau problem, I, Ann. of Math. 113 (1981), 67--110.
\end{thebibliography}
\end{document}